\documentclass{article}
\usepackage{amsfonts,amsmath,amsthm,latexsym,amscd,graphicx,xypic,psfrag}
\usepackage{epic, eepic}
\usepackage{bbold, mathabx}
\usepackage{verbatim}
\usepackage{setspace, enumitem}
\usepackage[utf8]{inputenc}
\usepackage[colorlinks=true]{hyperref}
\usepackage{mathtools}

\def\<#1,#2>{\langle #1,#2 \rangle}

 \def\ID{{\mathsurround0pt\mathchoice{\textID}{\textID}{\scptID}{\scptID}}}
 \def\scptID{\setbox0\hbox{$\scriptstyle1$}\bothID}
 \def\textID{\setbox0\hbox{$1$}\bothID}
 \def\bothID{\rlap{\hbox to.97\wd0{\hss\vrule height.06\ht0 width.82\wd0}}
 \copy0\rlap{\kern-.36\wd0\vrule height1.05\ht0 width.05\ht0}\kern.14\wd0}

 \DeclareMathOperator{\exte}{ext}
\DeclareMathOperator{\inte}{int} \DeclareMathOperator{\Hom}{Hom}
 
\DeclareMathOperator{\Cl}{Cl}

\begin{document}

\title{The Dirac-Dolbeault Operator Approach to the Hodge Conjecture}
\author{Simone Farinelli
        \thanks{Simone Farinelli, Aum\"ulistrasse 20,
                CH-8906 Bonstetten, Switzerland, e-mail simone.farinelli@alumni.ethz.ch}
        }
\maketitle

\begin{abstract}
The Dirac-Dolbeault operator for a compact K\"ahler manifold is a special case of a Dirac operator.
The Green function for the Dirac Laplacian over a Riemannian manifold with boundary allows to express the values of the sections of the Dirac bundle
in terms of the values on the boundary, extending the mean value theorem of harmonic analysis.
Utilizing this representation and the Nash-Moser generalized
inverse function theorem we prove the existence of complex submanifolds of a complex projective manifold
satisfying globally a certain partial differential equation under a certain injectivity assumption. Next, we show
the existence of complex submanifolds whose fundamental classes span the rational Hodge classes, proving the Hodge conjecture
for complex projective manifolds.\\\\
\vspace{0.2cm}
\vspace{0.2cm}
\end{abstract}

\newtheorem{theorem}{Theorem}[section]
\newtheorem{proposition}[theorem]{Proposition}
\newtheorem{lemma}[theorem]{Lemma}
\newtheorem{corollary}[theorem]{Corollary}
\theoremstyle{definition}
\newtheorem{ex}{Example}[section]
\newtheorem{rem}{Remark}[section]
\newtheorem*{nota}{Notation}
\newtheorem{defi}{Definition}
\newtheorem{conjecture}{Conjecture}

\tableofcontents

\section{Introduction}
The Hodge conjecture attempts to build a bridge between complex differential geometry and algebraic geometry on K\"ahler manifolds.
More precisely, it postulates a connection between topology
(Betti cohomology classes, i.e. cohomology with rational coefficients), complex
geometry (Hodge decomposition theorem for the De Rham cohomology in terms of Dolbeault cohomologies) and algebraic geometry
(the algebraic projective subvarieties of a complex projective algebraic variety).\par
The conjecture was formulated by W. Hodge during the 1930s, when he studied  the De Rham cohomology for complex algebraic varieties.
Hodge presented it during the 1950 International Congress of Mathematicians, held in Cambridge, Massachusetts, (\cite{Ho52}).
Before that date it had received little attention by the mathematical community. The current statement reads as follows (cf. \cite{De06}):
\begin{conjecture}[{\bf Hodge}]\label{Hodge1}
Let $X$ be a projective non-singular (i.e. without isolated points) algebraic variety over $\mathbf{C}$ and, for any $k=0,\dots,n:=\dim_{\mathbf{C}} X$
the rational Hodge class of degree $2k$ on $X$ is defined as $\text{Hdg}^k(X,\mathbf{Q}):=H^{2k}(X,\mathbf{Q})\cap H^{k,k}(X,\mathbf{C})$.
Then, any Hodge class on $X$ is a rational linear combination of classes of algebraic cycles.
\end{conjecture}
In Hodge's original conjecture the coefficients were not rational but integer. This version of the conjecture was proven  false
by Atiyah and Hirzenbruch \cite{AtHi62} with a first counterexample. Totaro (\cite{To97}) reinterpreted their result
in the framework of cobordism and constructed many others. Hodge's conjecture is false in the category of K\"ahler manifolds,
as Grothendieck (\cite{Gr69}) and Zucker (\cite{Zu77}) have recognized. For example, it is possible to construct a K\"ahler manifold,
namely a $2$-dimensional complex torus $T^2$, whose only analytic submanifolds are isolated points and the torus itself. Hence,
the Hodge conjecture  cannot hold for $\text{Hdg}^1(T^2)$. Voisin (\cite{Vo02}) proved that even more relaxed versions of the Hodge conjecture
for K\"ahler manifolds, with fundamental classes replaced by Chern classes of vector bundles or by Chern classes of coherent sheaves on $X$, cannot hold true,
by proving that the Chern classes of coherent sheaves give strictly more Hodge classes than the Chern classes of vector bundles,
 and that the Chern classes of coherent sheaves are insufficient to generate all the Hodge classes. \par
In a nutshell, Hodge's conjecture postulates a characterization for
cohomology classes generated over $\mathbf{Q}$  (i.e. algebraic classes) by classes of algebraic subvarieties of a given
dimension of a complex projective manifold $X$, more precisely by rational cohomology classes of degree $2k$ which admit
de Rham representatives which are closed forms of type $(k,k)$ for the complex structure on $X$ (i.e. Hodge classes). Note that the
integration over a complex submanifold of dimension $n-k$ annihilates forms of type $(p,q)$ with $(p,q)\neq (n-k,n-k)$.\par

The first result on the Hodge conjecture is due to Lefschetz, who proved it for $2$-Hodge classes with integer coefficients in \cite{Lef24}.
Combined with the Hard Lefschetz theorem, (see \cite{Vo10}, page 148), formulated by Lefschetz in 1924 and proved by Hodge in 1941,
 it implies that the Hodge conjecture is true for Hodge classes of degree
$2n-2$, proving the Hodge conjecture when $\dim X\le 3$.
Cattani, Deligne and Kaplan provide positive evidence for the Hodge conjecture in \cite{CDK95},
showing roughly that Hodge classes behave in a family as if they were algebraic.\par
For a thorough treatment of Hodge theory and complex algebraic geometry see \cite{Vo10}. For the official statement of the Hodge conjecture
for the Clay Mathematics Institute see \cite{De06}. For the current state of the research and the possible generalizations of the conjecture see \cite{Vo11, Vo16}.
For a presentation of many specific known cases of the Hodge conjecture see \cite{Lew99}.\par

This paper is structured as follows. In Section 2 we review the definitions of complex projective algebraic varieties, Hodge classes, Dirac bundles, and Dirac operators,
showing that the Dirac-Dolbeault operator on a K\"ahler manifold is the Dirac operator for the antiholomorphic bundle, and the Hodge-Kodaira Laplacian is the Dirac Laplacian.
In Section 3 we study the Green function for the Dirac Laplacian on a compact Riemannian manifold with boundary, and prove a representation theorem  expressing
the values of the sections of the Dirac bundle over the interior in terms of the values on the boundary. This result holds true for the Hodge-Kodaira Laplacian
over a compact K\"ahler manifold. In Section 4 we first review the Nash-Moser generalized
inverse function theorem, applying it to our geometric set-up by proving the existence of complex submanifolds of a complex projective manifold
satisfying globally a certain partial differential equation under a certain injectivity assumption, leading to the following key result.

\begin{proposition}\label{PropKey}
Let $X$ be a $n$-dimensional complex projective manifold without boundary and $\omega\in\Omega^{n-1,n-1}(X,\mathbf{C})$
a representative of the cohomology class $[\omega]\in H^{n-1,n-1}(X,\mathbf{Q})$. Then, $[\bar{*}\omega]$ is in $H^{1,1}(X,\mathbf{Q})$ and a fundamental class
of a closed complex projective  submanifold of complex codimension $1$ (i.e. a complex hypersurface)  if and only if
there exist an atlas $\{(U_i,\Phi_i)_{i=0,K}\}$ of $X$ such that
\begin{equation}
\begin{split}
&\mathcal{F}^{\omega,\{U_i\}_{i=0,\dots,K}}:=\left\{\psi:=(\psi_0,\psi_1,\dots,\psi_K)|\;\psi_i:\Phi_i(U_i)\rightarrow\psi_i(\Phi_i(U_i))\text{ is a } \right.\\
&\qquad\qquad\qquad\qquad\left.\text{diffeomorphism for all }i=0,\dots,K\text{ and }DF^{\omega}(\psi,T\psi)\text{ is injective}\right.\}\neq\emptyset.
\end{split}
\end{equation}
See Lemma \ref{TL} for the definition of $F$.
\end{proposition}

In Section 5, by recursively applying this key result, we prove the existence of complex submanifolds of a complex projective manifold
whose fundamental classes span the rational Hodge classes in $\text{Hdg}^k(X)$ for all $k=0,\dots,n$.
This is a slightly stronger result than the Hodge conjecture for non singular projective algebraic varieties.
As expected, the presented proof cannot be extended to the category of K\"ahler manifolds, or to integer Hodge cohomology.
\section{Definitions}
We first review some standard facts about the complex projective space, K\"ahler manifolds, Dolbeault and Hodge cohomologies, and Dirac bundles, establishing the necessary notation.
\begin{defi}\label{defCPN} Let $n\in\mathbf{N}_1$. The \textbf{complex projective space} is the quotient space
\begin{equation}
\mathbf{C}P^n:=\left(\mathbf{C}^{n+1}\setminus\{0\}\right)/ \sim
\end{equation}
for the equivalence relation $\sim$ in $\mathbf{C}^{n+1}\setminus\{0\}$, defined as
\begin{equation}
a\sim b:\Leftrightarrow \exists \lambda\in\mathbf{C}:\,a=\lambda b,
\end{equation}
for $a,b\in\mathbf{C}^{n+1}\setminus\{0\}$. The quotient map
\begin{equation}
\begin{split}
q:&\mathbf{C}^{n+1}\longrightarrow \mathbf{C}P^n\\
&a\longmapsto q(a):=[a]
\end{split}
\end{equation}
induces an holomorphic atlas $\{(U_i,\Phi_i)\}_{i=0,\dots,n}$ on $\mathbf{C}^{n+1}$ given by
\begin{equation}
\begin{split}
\Phi_i:&U_i\longrightarrow \mathbf{C}^n\\
&[a]\longmapsto \Phi_i([a]):=\left(\frac{a_1}{a_i},\dots,\frac{a_{i-1}}{a_{i}},\frac{a_{i+1}}{a_{i}},\dots,\frac{a_n}{a_i}\right)
\end{split}
\end{equation}
for the open set $U_i:=\left\{[a]\in\mathbf{C}^{n+1}\in|\,a_i\neq0\right\}$. Any $a\in\mathbf{C}^n$ is mapped to a point in $\mathbf{C}P^n$ identified by its
\textbf{homogeneous coordinates}
\begin{equation}
\Phi_i^{-1}(a)=[a_1,\dots,a_{i-1},1,a_{i+1},\dots,a_n].
\end{equation}
For any $i,j=0\dots, n$ the \textbf{change of coordinate} maps $\Phi_i^{-1}\circ\Phi_j:\mathbf{C}\rightarrow\mathbf{C}$ is biholomorphic and the
complex projective space has thus the structure of a complex manifold without boundary.
\end{defi}
\begin{proposition}\label{cext} A compact complex manifold $X$ of complex dimension $n$ has a finite atlas $(V_k, \Psi_k)_{k=0,\dots,K}$
such that for every $k$ the set $\overline{V}_k$ is compact
and every chart $\Psi_k:V_k\rightarrow \Psi_k(V_k)\subset \mathbf{C}^n$ has a continuous extension
$\overline{\Psi}_k:\overline{V}_k\rightarrow \overline{\Psi}_k(\overline{V}_k)\subset \mathbf{C}^n$
with image in a compact subset of $\mathbf{C}^n$.
\end{proposition}
\begin{proof}
It suffices to refine any finite atlas $(U_i,\Phi_i)_{i=0,\dots,n}$, which exists because $X$ is compact.
Every $U_i$ can be represented as the union of open subsets of $U_i$, such that their closure is still contained in $U_i$:
\begin{equation}
U_i=\bigcup_{\substack{V \text{open}\\V\subset U_i\\\overline{V}\subset U_i}}V,
\end{equation}
Those $V$s form an open cover of $X$. Since $X$ is compact, it must exist a finite subcover $\{V_{k}\}_{k=0,\dots, K}$ for a $K\in\mathbf{n}_0$.
The closure $\overline{V_{k}}$ is compact and is the domain of the continuous extension $\overline{\Psi}_k$ of the well defined
$\Psi_k:=\Phi_{i}|_{V_{k}}$, for $V_k\subset U_i$.\\
\end{proof}
\begin{defi}\label{sub}
A \textbf{complex/real analytic/differentiable submanifold} $Y$ of complex/real/real dimension $m$ of a complex/real analytic/differentiable manifold
 $X$ of complex/real/real dimension $n$
is a subset $Y\subset X$ such that for the atlas
$(U_\iota,\Phi_\iota)_{\iota\in I}$ of $X$ there exist analytic/real analytic/differentiable homeomorphisms
$(\varphi_\iota:\mathbf{C}^n/\mathbf{R}^n/\mathbf{R}^n\hookrightarrow\mathbf{C}^n/\mathbf{R}^n/\mathbf{R}^n)_{\iota\in I}$ such that
\begin{equation}\label{phi}
\varphi_\iota(\Phi_\iota(U_\iota\cap Y))\subset\mathbf{C}^{m}\times\{0\}^{n-m}/\mathbf{R}^{m}\times\{0\}^{n-m}/\mathbf{R}^{m}\times\{0\}^{ n-m}
\end{equation}
for all $\iota\in I$. The \textbf{compatibility condition} reads
\begin{equation}
\Phi_\iota^{-1}\circ\varphi_\iota^{-1}|_{\Phi_\iota\circ\varphi_\iota(U_\iota\cap U_\kappa)}=\Phi_\kappa^{-1}\circ\varphi_\kappa^{-1}|_{\Phi_\kappa\circ\varphi_\kappa(U_\kappa\cap U_\iota)},
\end{equation}
for all $\iota, \kappa\in I$. The subset $Y$ of $X$ is a complex/real analytic/differentiable manifold complex/real/real dimension $m$
with atlas $(U_\iota\cap Y, \Pi\circ\varphi_\iota\circ\Phi_\iota)_{\iota\in I}$,
where $\Pi:\mathbf{C}^n/\mathbf{R}^n/\mathbf{R}^n\rightarrow\mathbf{C}^m/\mathbf{R}^m/\mathbf{R}^m$ denotes the projection onto the first $m$ dimensions
 of $\mathbf{C}^n/\mathbf{R}^n/\mathbf{R}^n$.
\end{defi}
It is possible to define submanifolds of a manifold by specifying appropriate change of coordinate maps
\begin{proposition}\label{propdefsub}
Let $X$ be a complex/real analytic/differentiable manifold of complex/real/real dimension $n$ with atlas $(U_\iota,\Phi_{\iota})$.
The analytic/real analytic/differentiable local homeomorphisms
\begin{equation}
(\varphi_\iota:\mathbf{C}^{n}/\mathbf{R}^{n}/\mathbf{R}^{n}\hookrightarrow\mathbf{C}^{n}/\mathbf{R}^{n}/\mathbf{R}^{n})_{\iota\in I}
\end{equation}
define a analytic/real analytic/differentiable submanifold $Y$ of $X$ of complex/real/real dimension $m$  by
\begin{equation}
Y:=\bigcup_{\iota\in I}\Phi_{\iota}^{-1}\circ\varphi_{\iota}^{-1}(V_\iota)
\end{equation}
for $V_{\iota}:=\varphi_\iota(\Phi_\iota(U_\iota))\cap\left(\mathbf{C}^{m}\times\{0\}^{n-m}/\mathbf{R}^{m}\times\{0\}^{n-m}/\mathbf{R}^{m}\times\{0\}^{n-m}\right)$\\
 if and only if for all $\iota,\kappa\in I$
\begin{equation}
\Phi_{\iota}^{-1}\circ\varphi_{\iota}^{-1}|_{V_\iota\cap V_\kappa}=\Phi_{\kappa}^{-1}\circ\varphi_{\kappa}^{-1}|_{V_\kappa\cap V_\iota}.
\end{equation}
\end{proposition}
\begin{proof}
It suffices to prove that the compatibility condition is satisfied.
\end{proof}
\noindent Following the clear and concise exposition of chapter 1 in \cite{Pe95} we have
\begin{defi}
If $Y$ is a $m$-dimensional complex submanifold of the $n$ dimensional complex manifold $X$,
then the Jacobian of the defining functions $\varphi_\iota$ in (\ref{phi}) is constantly equal to $m$ for all charts.
If we drop the condition about the Jacobian, Y is termed \textbf{
analytic subset} of $X$, which is called \textbf{irreducible} if it is not the union of non-empty smaller
analytic subsets. An irreducible analytic subset is also called an \textbf{analytic subvariety} and
the terms smooth subvariety and non-singular subvariety mean the same as complex submanifold.
\end{defi}

\begin{defi}An \textbf{affine algebraic set} is the zero set of a collection of polynomials. An \textbf{affine variety} is an irreducible affine algebraic set,
i.e. an affine algebraic set which cannot be written as the union of two proper algebraic subsets. 
A \textbf{projective algebraic set} is the zero set of a collection of homogenous polynomials, and can be seen as
a subset of the complex projective space $\mathbf{C}P^n$ for some $n\in\mathbf{N}_1$.
A \textbf{projective algebraic variety} is an irreducible projective algebraic set.
If it is a complex submanifold of the complex projective space, it is termed \textbf{projective manifold}.
\end{defi}

\begin{rem}
A complex manifold is orientable. A complex projective manifold is orientable and compact.
\end{rem}
\begin{rem} On the complex projective space we consider homogeneous polynomials of degree $d$ for any $d\in\mathbf{N}_0$. The evaluation of a polynomial
is not well defined on $\mathbf{C}P^n$, but, if it is homogeneous, its zero set is.
\end{rem}

\begin{defi}
Let $X$ be a projective variety. An \textbf{analytic $k$-cycle} is a formal linear combination
\begin{equation}
\sum_jc_jY_j,
\end{equation}
where $\{Y_j\}_j$ is a collection of $k$-dimensional closed irreducible analytic subsets of $X$, and $(c_j)\subset\mathbf{Z}$ for \textbf{integral analytic cycles} and
$(c_j)\subset\mathbf{Q}$ for \textbf{rational analytic cycles}.
\end{defi}

\begin{theorem}[\textbf{Chow}]
Any analytic subvariety of the complex projective space is a projective variety.
\end{theorem}
\begin{proof} See \cite{Mu76}.
\end{proof}

\begin{rem}\label{remChow}
By Chow's theorem on a complex projective variety $X$, the algebraic subsets of $X$ are exactly the analytic subsets of $X$, and we do not need
to distinguish between \textbf{algebraic and analytic cycles}, see \cite{De06} and \cite{Vo02} page 272. If there are no singularities, then
$\{Y_j\}_j$ is a collection of $k$-dimensional complex submanifolds of $X$.
\end{rem}

\begin{corollary}\label{subvar}
Any complex submanifold of a projective manifold is a projective (sub)manifold.
\end{corollary}
\begin{defi}
The quadruple $(V,\<\cdot,\cdot>,\nabla,\gamma)$, where
  \begin{enumerate}
  \item[(1)] $V$ is a complex (real) vector bundle over the Riemannian manifold $(X,g)$ with Hermitian (Riemannian) structure $\<\cdot,\cdot>$,
  \item[(2)] $\nabla:C^\infty(X,V)\to C^\infty(X,T^*X\otimes V)$ is a connection on $X$,
  \item[(3)] $\gamma: \Cl(X,g)\to \Hom(V)$ is a {\it real} algebra bundle homomorphism from the Clifford bundle over $X$ to the {\it real} bundle of complex (real) endomorphisms of $V$, i.e. $V$ is a bundle of Clifford modules,
 \end{enumerate}
    is said to be a {\bf Dirac bundle}, if the following conditions are satisfied:
    \begin{enumerate}
    \item[(4)] $\gamma(v)^*=-\gamma(v)$, $\forall v\in TX$ i.e. the Clifford multiplication by tangent vectors is fiberwise skew-adjoint with respect to the Hermitian (Riemannian) structure $\<\cdot,\cdot>$.
    \item[(5)] $\nabla\<\cdot,\cdot>=0$ i.e. the connection is Leibnizian (Riemannian). In other words it satisfies the product rule:
      \begin{equation*}
        d\<\varphi,\psi>=\<\nabla\varphi,\psi>+\<\varphi,\nabla\psi>,\quad\forall \varphi,\psi\in C^\infty(X,V).
      \end{equation*}
    \item[(6)] $\nabla\gamma=0$ i.e. the connection is a module derivation. In other words it satisfies the product rule:
      \begin{equation*}
        \nabla(\gamma(w)\varphi)=\gamma(\nabla^gw)\varphi+\gamma(w)\nabla\varphi,\quad\forall \varphi,\psi\in C^\infty(X,V),
\,\forall w\in C^\infty(X,\Cl(X,g)).
\end{equation*}
    \end{enumerate}
    The {\bf Dirac operator} $Q:C^\infty(X,V)\to C^\infty(X,V)$ is defined by 
\[\begin{CD}
   {C^{\infty}(X,V)} @>{\nabla}>> {C^{\infty}(X,T^*X\otimes V)} \\
    @V{Q:=\gamma\circ(\sharp\otimes\ID)\circ\nabla}VV  @VV{\sharp\otimes\ID}V \\
   {C^{\infty}(X,V)} @<{\gamma}<< {C^{\infty}(X,TX\otimes V)}
 \end{CD}\]
 and its square $P:=Q^2:C^\infty(X,V)\to C^\infty(X,V)$ is called the {\bf Dirac Laplacian}.
\end{defi}
\begin{defi}\label{defK} A \textbf{K\"ahler manifold} is a Riemannian manifold $(X,g)$ of even real dimension $2n$ such that there exists an almost complex structure $J$ on $TX$, that is
$J_x:T_xX\rightarrow T_xX$, for all $x\in X$, real linear with $J^2=-\mathbb{1}$, for which $g(Ju,Jv)=g(u,v)$ and $J$ is preserved by the parallel transport induced by the Levi-Civita connection
$\nabla^g$. The symplectic closed two form $w(U,V):=g(U,JV)$ is called \textbf{K\"ahler form}.
\end{defi}

\begin{rem}\label{remFS}
The complex projective space  carries a (K\"ahler) metric, called the \textbf{Fubini–Study metric}, which in homogeneous coordinates reads (see f.i.\cite{DjOk10} chapter 4)
\begin{equation}\label{FS}
g^{\text{FS}}([z])(A,B)=\sum_{h,i=1}^n\left[\frac{\left(1+\sum_{k=1}^n\vert t^k_j\vert^2\right)\delta_{h,i}-t^i_j\bar{t}^h_j}{\left(1+\sum_{k=1}^n\vert t^k_j\vert^2\right)^2}\right](\alpha^h\overline{\beta}^i+\overline{\alpha}^i\beta^h),
\end{equation}
where
\begin{itemize}
\item $[z]\in U_j:=\{[(z^0,\dots,z^{j-1}_j,z^j,z^{j+1},\dots,z^{n})]\in\mathbf{C}P^n\vert\,z^j\neq0\}$ is a generic point in the complex projective space
and $[z]$ has homogeneous coordinates on $U_j$ given by
$t^k_j:=\frac{z^k}{z^j}$ for $k\neq j$.
\item $A,B$ are tangential vectors to the complex projective space given by
\begin{equation}
A=\sum_{i=1}^n\left[\alpha^i\frac{\partial}{\partial t^j_i}+\overline{\alpha}^i\frac{\partial}{\partial \overline{t}^j_i}\right]\qquad
B=\sum_{i=1}^n\left[\beta^i\frac{\partial}{\partial t^j_i}+\overline{\beta}^i\frac{\partial}{\partial \overline{t}^j_i}\right].
\end{equation}
\end{itemize}
All complex submanifolds of $\mathbf{C}P^n$ are examples of K\"ahler manifolds.
\end{rem}

\begin{proposition}
A complex submanifold $Y$ of a K\"ahler manifold $X$ is K\"ahler.
\end{proposition}
\begin{proof}
See \cite{Bal06,Mo10}.
\end{proof}

\begin{proposition}[\textbf{Wirtinger's formula}]
Let $X$ be a K\"ahler manifold with K\"ahler form $w$. For any $m$-dimensional complex submanifold $Y\subset X$, the volume form of $Y$ satisfies
\begin{equation}
\mu_Y=i_Y^*\left(\frac{w^{\wedge m}}{m!}\right).
\end{equation}
\end{proposition}
\begin{proof}
See \cite{GrHa94}, page 31.
\end{proof}

\begin{proposition}\label{prop21}{\bf (Antiholomorphic Bundle as a Dirac Bundle).}
Let $(X,g, J)$ be a K\"ahler manifold of  real dimension $2n$
with Riemannian metric $g$ and
almost complex structure $J\in\Hom(TX)$ satisfying $J^2=-\mathbb{1}$.  The antiholomorphic bundle can be
seen as a Dirac bundle $(V,\langle,\cdot,\cdot\rangle,\nabla)$ with the following choices:
\begin{itemize}
\item $V:=\Lambda(T^{0,1}X)^*$: antiholomorphic bundle over $X$.
\item $\langle\cdot,\cdot\rangle:=g^{\Lambda(T^{0,1}X)^*}$.
\item $\nabla:=\nabla^{g^{\Lambda(T^{0,1}X)^*}}$.
\item By means of interior and exterior multiplication, by utilizing the decomposition of $TX$ with respect to the $\pm\imath$-eigenspaces of $J$, we can define
 \begin{equation}
        \begin{split}
        \gamma&:
            \begin{array}{lll}
                TX=TX^{1,0}\oplus TX^{0,1}&\longrightarrow&\Hom(V)\\
                v=v^{1,0}\oplus
                v^{0,1}&\longmapsto&\gamma(v):=\sqrt{2}(\exte(v^{1,0})-\inte(v^{0,1})).
            \end{array}
        \end{split}
    \end{equation}
     Since $\gamma^2(v)=-g(v,v)\mathbb{1}$, by the universal property, the map $\gamma$ extends uniquely to a real algebra bundle endomorphism $\gamma:\Cl(X,g)\longrightarrow\Hom(V)$.
\end{itemize}
The \textbf{Dolbeault operators} $\partial$ and $\bar{\partial}$ have formal adjoints satisfying
$\partial^*=-\bar{*}\partial\bar{*}$ and $\bar{\partial}^*=-\bar{*}\bar{\partial}\bar{*}$,
where $\bar{*}$ is the \textbf{conjugate-linear Hodge star operator} fulfilling $\bar{*}\bar{*}=(-1)^{p+q}$ on $\Omega^{p,q}(X,\mathbf{C})$.
The Dirac operator $Q$ in the case of antiholomorphic bundles over
K\"ahler manifolds $(X,g, \Omega, J)$ is the \textbf{Dirac-Dolbeault operator}
$\sqrt{2}(\overline{\partial}+\overline{\partial}^*)$, while the
Dirac Laplacian $P:=Q^2$ is the \textbf{Hodge-Kodaira Laplacian}
$\Delta_{\bar{\partial}}:=2(\overline{\partial}\overline{\partial}^*+\overline{\partial}^*\overline{\partial})$.\par
The cohomology group of $X$ with complex coefficients lie in degrees $0$ through $2n$ and there is a decomposition
\begin{equation}
H^k(X,\mathbf{C})=\bigoplus_{p+q=k}H^{p,q}(X,\mathbf{C}),
\end{equation}
where $H^{p,q}(X,\mathbf{C})$ is the subgroup of cohomology classes represented by harmonic forms of type $(p,q)$, termed \textbf{Dolbeault cohomology}.
\end{proposition}
\begin{proof}
See f.i. Chapters 3.5 and 3.6 of \cite{Gi84}.
\end{proof}

\begin{theorem}[\textbf{Lefschetz Decomposition on Cohomology}]\label{Lefschetz} Let $X$ be a $n$ complex dimensional compact K\"ahler manifold with K\"ahler form $w$, and
for any $k=0,\dots,2n-2$
\begin{equation}
\begin{split}
L:\Omega^k(X,\mathbf{C})&\rightarrow\Omega^{k+2}(X,\mathbf{C})\\
\alpha&\mapsto L\alpha:=w\wedge\alpha.
\end{split}
\end{equation}
Then, $L$ defines an operator
\begin{equation}
\begin{split}
L:H^k(X,\mathbf{C})&\rightarrow H^{k+2}(X,\mathbf{C})\\
[\alpha]&\mapsto L[\alpha]:=[w\wedge\alpha],
\end{split}
\end{equation}
such that, for any $r\le n$
\begin{equation}
L^r:H^k(X,\mathbf{C})\rightarrow H^{k+2r}(X,\mathbf{C})
\end{equation}
is an isomorphism. Moreover, every cohomology class $[\alpha]\in H^k(X,\mathbf{C})$ admits a unique decomposition
\begin{equation}
[\alpha]=\sum_rL^r[\alpha_r],
\end{equation}
where $\alpha_r$ is of degree $k-2r\le\min\left(n,2n-k\right)$ and $L^{n-k+2r+2}[\alpha_r]=[0]\in H^{2n-k+2r+1}(X,\mathbf{C})$.
\end{theorem}
\begin{proof} See Theorem 6.25, Corollary 6.26 and Remark 6.27 in \cite{Vo10}.
\end{proof}
\begin{proposition}\label{starrep}
With the same assumptions as Theorem \ref{Lefschetz}, for $\alpha\in\Omega^{p,q}(X,\mathbf{C})$ such that $L^{n-k+1}\alpha=0$ for $k:=p+q$, then
\begin{equation}
\bar{*}\alpha=(-1)^{\frac{k(k+1)}{2}}\imath^{p-q}\frac{L^{n-k}}{\left(n-k\right)!}\alpha.
\end{equation}
\end{proposition}
\begin{proof}
See Proposition 6.29 in \cite{Vo10}.
\end{proof}

\begin{defi} If the compact K\"ahler manifold $X$ is boundaryless, by De Rham's theorem, we can define a scalar product for $H^{k,k}(X,\mathbf{C})$ by means of the expression
\begin{equation}\label{sp}
([\alpha], [\omega]):=\int_X\alpha\wedge\overline{\ast}\omega,
\end{equation}
where $\alpha,\omega\in\ker(\Delta^{k,k}_{\bar{\partial}})$ are the unique harmonic representatives for the cohomology classes
$[\alpha],[\omega]\in H^{k,k}(X,\mathbf{C})\cong\ker(\Delta^{k,k}_{\bar{\partial}})$. By Riesz's Lemma,
\begin{equation}
H^{k,k}(X,\mathbf{C})^*=H^{k,k}(X,\mathbf{C}),
\end{equation}
where the isomorphism is induced by the scalar product in (\ref{sp})
\begin{equation}
\Omega^{k,k}(X,\mathbf{C})\mathrel{\mathop{\rightleftarrows}^{\mathrm{\flat}}_{\mathrm{\sharp}}}\Omega^{k,k}(X,\mathbf{C})^*\qquad
  H^{k,k}(X,\mathbf{C})\mathrel{\mathop{\rightleftarrows}^{\mathrm{\flat}}_{\mathrm{\sharp}}} H^{k,k}(X,\mathbf{C})^*.
\end{equation}
\noindent  The map
\begin{equation}
i_Z:Z\rightarrow X
\end{equation}
denotes the embedding of any complex submanifold $Z$ into $X$ and
\begin{equation}
\begin{split}
i_Z^*:\Omega^{k,k}(X,\mathbf{C})&\rightarrow \Omega^{k,k}(Z,\mathbf{C})\\
\alpha&\mapsto i_Z^*\alpha:=\alpha(\underbrace{Ti_Z.(\cdot),\dots,Ti_Z.(\cdot)}_{2k\text{ times}})
\end{split}
\end{equation}
the \textbf{pull back} of $(k,k)$-forms on $X$ on $(k,k)$-forms on $Z$.
\end{defi}

\begin{rem} For any topological space $X$ and for $\mathbf{F}\in\{\mathbf{Z},\mathbf{Q},\mathbf{R},\mathbf{C}\}$, the \textbf{singular $\mathbf{F}$-homology} $H_p(X,\mathbf{F})$ is the homology of
the $\mathbf{F}$-chains, and the \textbf{singular $\mathbf{F}$-cohomology} $H^p(X,\mathbf{F})$ is the homology of
the $\mathbf{F}$-cochains, see \cite{BoTu82}, chapter III.15. If $X$ is a differentiable real manifold, by the De Rham theorem the singular cohomology and the De Rham cohomology with real or complex cofficients are
isomorphic and for $\omega\in\Omega^p(X,\mathbf{C})$
\begin{equation}\label{char}
[\omega]\in H^p(X,\mathbf{F})\Longleftrightarrow\int_Yi_Y^*\omega\in\mathbf{F}\text{ for all oriented real submanifolds }Y\subset X\text{ such that }[Y]\in H_p(X,\mathbf{F}),
\end{equation}
where $i_Y$ is the injection of $Y$ in $X$. Note that $H^p(X,\mathbf{F})$ in (\ref{char})
denotes the image of the singular cohomology in the De Rham cohomology. If $X$ is a complex manifold, any complex submanifold $Y$ has a natural orientation.
\end{rem}

\begin{defi} Let $X$ be a K\"ahler manifold of complex dimension $n$. For $k=0,\dots,n$ the \textbf{rational Hodge class of degree $2k$} on $X$ is defined as
\begin{equation}
\text{Hdg}^k(X,\mathbf{Q}):=H^{2k}(X,\mathbf{Q})\cap H^{k,k}(X,\mathbf{C}).
\end{equation}
For $k=0,\dots,n$ the \textbf{integer Hodge class of degree $2k$} on $X$ is defined as
\begin{equation}
\text{Hdg}^k(X,\mathbf{Z}):=H^{2k}(X,\mathbf{Z})\cap H^{k,k}(X,\mathbf{C}).
\end{equation}
\end{defi}

\begin{theorem}[\textbf{Kodaira’s criterion}]\label{Kod} A compact complex manifold $X$ is projective
if and only if $X$ admits an integer K\"ahler class $[w]$, that is, belonging to $H^2(X,\mathbf{Z})$.
\end{theorem}
\begin{proof}
See \cite{Ko54} or \cite{We08} Theorem VI.4.1 and Example VI.1.2.
\end{proof}
\begin{proposition}\label{close}
Let $X$ be a K\"ahler manifold of complex dimension $n$ and $\mathbf{F}\in\{\mathbf{Z},\mathbf{Q},\mathbf{R},\mathbf{C}\}$.
For $k, l\in\{0,\dots,n\}$ and $[\omega]\in H^{k,k}(X,\mathbf{F})$, $[\eta]\in H^{l,l}(X,\mathbf{F})$, $f_n(j):=j1_{\{0,\dots,n\}}(j)$
\begin{equation}\label{wedgeref}
[\omega\wedge\eta]\in H^{f_n(k+l), f_n(k+l)}(X,\mathbf{F})
\end{equation}
holds true. Moreover, if $X$ is a complex projective manifold and $\mathbf{F}\in\{\mathbf{Q},\mathbf{R},\mathbf{C}\}$, we have
\begin{equation}\label{starref}
[\bar{*}\omega]\in H^{n-k, n-k}(X,\mathbf{F}).
\end{equation}
\end{proposition}
\begin{proof}
For $\mathbf{F}=\mathbf{C}$ the statement (\ref{wedgeref}) follows from the fact that the Dolbeault operators are antiderivations
and the second from Serre's duality (see \cite{Gi84}, page 199). This is true for $\mathbf{F}\in\{\mathbf{Z},\mathbf{Q}, \mathbf{R}\}$ as well, but
we have to additionally to prove that
for all $f_n(k+l)$ dimensional complex submanifolds $Y\subset X$ such that $[Y]\in H_{f_n(k+l)}(X,\mathbf{F})$
\begin{equation}\label{wedge}
\int_Y i_Y^*(\omega\wedge\eta)\in \mathbf{F},
\end{equation}
and for all $(n-k)$ dimensional complex  submanifolds $Y\subset X$ such that $[Y]\in H_{2(n-k)}(X,\mathbf{F})$
\begin{equation}\label{star}
\int_Y i_Y^*(\bar{*}\omega)\in \mathbf{F},
\end{equation}
since for $\mathbf{F}=\mathbf{C}$ these statements are trivially true.
We begin with (\ref{wedge})
and consider the diagonal map $\mathfrak{d}:X\rightarrow X\times X$, denoting with $\pi_1$ and $\pi_2$ the projections from $X\times X$
 onto its first and second factor. For a $\mathbf{F}$-$(k+l)$-cycle $Y$ we have
\begin{equation}
\int_{Y}i_Y^*(\omega\wedge\eta)=\int_{\mathfrak{d}(Y)}\pi_1^*(\omega)\wedge\pi_2^*(\eta).
\end{equation}
Let us suppose that $\mathfrak{d}(Y)$ were homologous in $X\times X$ to $\sum_i S_i\times T_i$, for various cycles $S_i$ and $ T_i$ in $H_{*}(X,\mathbf{F})$,
 with $\dim(S_i)+\dim( T_i)=k+l$. Then we would have
\begin{equation}\label{int_wedge}
\int_{Y}i_Y^*(\omega\wedge\eta)=\sum_i\int_{S_i\times T_i}\pi_1^*(\omega)\wedge\pi_2^*(\eta)=\sum_{(\dim(S_i),\dim( T_i))=(k,l)}\int_{S_i}i_{S_i}^*(\omega)\int_{T_i}i_{T_i}^*(\eta).
\end{equation}
This would prove the result, because the integrals over terms where $\dim(S_i)\neq k$ would drop out.
If $\dim(S_i)<k$, then $\pi_1^*(\omega)\vert_{S_i\times T_i}=0$ so $\int_{S_i\times T_i}\pi_1^*(\omega)\wedge\pi_2^*(\eta)=0$,
 and likewise if $\dim(T_i)<l$.
%
%
K\"unneth's theorem includes the statement that
\begin{equation}
0\rightarrow \bigoplus_{i+j=2n}H_i(X,\mathbf{F})\otimes_{\mathbf{F}} H_j(X,\mathbf{F})\rightarrow H_{2n}(X\times X,\mathbf{F})\rightarrow\bigoplus_{i+j=2n}\text{Tor}_1^{\mathbf{F}}(H_i(X,\mathbf{F}),H_j(X,\mathbf{F}))\rightarrow 0
\end{equation}
is (noncanonically) split, where $\text{Tor}_1^{\mathbf{F}}$ is the first Tor functor. If we choose such a splitting
\begin{equation}
\iota:\bigoplus_{i+j=2n}\text{Tor}_1^{\mathbf{F}}(H_i(X,\mathbf{F}),H_j(X,\mathbf{F}))\rightarrow H_{2n}(X,\mathbf{F}),
\end{equation}
then we can write $\mathfrak{d}(Y)=\sum_iS_i\times T_i+\iota\left(\sum_j\vartheta_j\right)$ for some $\vartheta_j$ in torsion groups. So, any $\mathbf{F}$-multiple of $\vartheta_j$ is homologous to zero, which means that
\begin{equation}
\int_{\iota(\vartheta_j)}\pi_1^*(\omega)\wedge\pi_2^*(\eta)=0.
\end{equation}
 So, we conclude that
\begin{equation}\label{wedgefree}
\int_{Y}i_Y^*(\omega\wedge\eta)=\sum_{i}\int_{S_i\times T_i}\pi_1^*(\omega)\wedge\pi_2^*(\eta).
\end{equation}
and, utilizing (\ref{int_wedge}), equation (\ref{wedge}) is proved.\\
We show now that statement (\ref{starref}) holds true. For $\mathbf{F}=\mathbf{R},\mathbf{C}$ it is evident for any K\"ahler manifold $X$.
Any $[\omega]\in H^{k,k}(X,\mathbf{C})$ can be decomposed by Theorem \ref{Lefschetz} as
\begin{equation}
[\omega]=\sum_rL^r[\omega_r],
\end{equation}
where $\omega_r$ is of degree $2k-2r\le\min\left(n,2n-2k\right)$ and $L^{n-2k+2r+1}[\omega_r]=[0]\in H^{n+1}(X,\mathbf{C})$.
We then apply the Hodge star operator to obtain
\begin{equation}
\bar{*}[\omega]=\sum_r\bar{*}L^r[\omega_r].
\end{equation}
By applying Proposition \ref{starrep}, since $L^{n-2k+2r+1}[\omega_r]=0$, we see that
\begin{equation}
\bar{*}L^r[\omega_r]=(-1)^{\frac{(2k-2r)(2k-2r_1)}{2}}\frac{L^{n-2k+3r}}{(n-2k+2r)!}[\omega_r],
\end{equation}
which does not vanish if and only if $r=0$. Therefore,
\begin{equation}
\bar{*}[\omega] =\bar{*}[\omega_0]\text{ and }
[\omega] = [\omega_0].
\end{equation}
By Kodaira's criterion (Theorem \ref{Kod}), the K\"ahler class is integer, because $X$ is complex projective,
and by (\ref{wedgeref}) $L^{n-2k+3r}[\omega_0]\in H^{2}(X,\mathbf{Z})$.
We conclude that $\bar{*}[\omega]\in H^{2}(X,\mathbf{Q})$, so that (\ref{starref}) is proved.\\
\end{proof}

\begin{rem} Actually, for the proof of the Hodge conjecture we only need to study the case $\mathbf{F}=\mathbf{Q}$. Nevertheless,
we verify the truth of all needed partial results for all choices of $\mathbf{F}$ in order to understand why
the proof of the Hodge conjecture works for $\mathbf{F}=\mathbf{Q}$ but does not for $\mathbf{F}=\mathbf{Z,R,C}$.
\end{rem}

\begin{defi}\label{deffc}
Let $Z$ be a $k$-complex codimensional closed submanifold of the $n$ complex dimensional K\"ahler manifold $X$. The expression
\begin{equation}\label{defZ}
[Z]:=\bar{*}\left[\left(\int_Zi_Z^*(\cdot)\right)^{\sharp}\right]
\end{equation}
defines a cohomology class in $H^{k,k}(X,\mathbf{C})$ by $Z$, which is termed the \textbf{fundamental class}.
\end{defi}
\begin{rem} The definition of $[Z]$ carries over for any closed differentiable real $k$-codimensional submanifold $Z$ of $X$ for any closed real $n$-dimensional differentiable
manifold $X$, utilizing the real Hodge star operator $*$, and the pull-back of $i_Z^*:\Omega^{k}(X,\mathbf{R})\rightarrow \Omega^{k}(Z,\mathbf{R})$,
leading to a $[Z]\in H^k(X,\mathbf{R})$.
\end{rem}
\begin{proposition}
Let $X$ be a compact complex manifold without boundary and $Z$ a complex $k$-codimensional submanifold of $X$.
 The cohomology class $[Z]$ defined by the expression (\ref{defZ}) satisfies
\begin{equation}
[Z]=J_Z(T^{-1}(1)),
\end{equation}
where $T:H^k(X,X\setminus Z,\mathbf{Z})\cong H^0(Z,\mathbf{Z})$ is the Thom isomorphism and
$j_Z:H^k(X,X\setminus Z,\mathbf{Z})\rightarrow H^k(X,\mathbf{Z})$ the natural map.
In particular, if $X$ is a K\"ahler manifold, every fundamental class belongs to the integer Hodge cohomology $\text{Hdg}^k(X,\mathbf{Z})$.
\end{proposition}
\begin{proof}
It is a reformulation of Corollary 11.15 in \cite{Vo10} page 271.\\
\end{proof}
\noindent By  Remark \ref{remChow} and Definition \ref{deffc} the Hodge conjecture can  be restated as
\begin{conjecture}[{\bf Hodge}]\label{Hodge2}
On a non-singular complex projective manifold $X$ any rational Hodge class is a rational linear combination of the fundamental classes
of closed complex subvarieties of $X$.
\end{conjecture}

\section{Green Function for the Dirac Laplacian}
\begin{defi} A \textbf{Green function} for the Dirac Laplacian $P$ on the Dirac bundle $(V,\langle,\cdot,\cdot\rangle,\nabla)$ over the
Riemannian manifold $(X,g)$ under the \textbf{Dirichlet boundary condition}
is given by the smooth section
\begin{equation}
G:X\times X\setminus\Delta\rightarrow V\boxtimes V^*,
\end{equation}
locally integrable in $X\times X$, which satisfies in the weak (i.e. distributional) sense, the following boundary problem:
\begin{equation}
\left\{
  \begin{array}{l}
    P_yG(x,y)=\delta(y-x)\mathbb{1}_{V_y}\\
    G(x,y)= 0,\;\text{ for } x\in\partial X\setminus\{y\},
  \end{array}
\right.
\end{equation}
for all $x\neq y\in X$, where $\Delta:=\{(x,y)\in X\times X|\,x=y\}$, and $V\boxtimes V^*$ is the fibre bundle over $X\times X$ such that the fibre over $(x,y)$ is given by
$\text{Hom}(V_y, V_x)$.\\ In other words, we have
\begin{equation}
\int_X\langle G(x,y)\psi(x),P_y\varphi(y)\rangle\, d\text{vol}_{y\in X}=\langle \psi(x),\varphi(x)\rangle_x,
\end{equation}
for all $x\in X$ and all sections $\psi,\varphi\in C^{\infty}(X,V)$, where $\varphi$ satisfies $\varphi|_{\partial X}=0$, the Dirichlet boundary condition .
\end{defi}
\begin{proposition}\label{prop31}
Let $(V,\langle,\cdot,\cdot\rangle,\nabla)$ be a Dirac bundle over the compact Riemannian manifold $X$. Then, the Dirac Laplacian $P$ has a Green function under the Dirichlet
boundary condition.
\end{proposition}
\begin{proof} The proof formally follows the steps of the proof for the Laplace Beltrami operator as in chapter $4$ of \cite{Au82}. See \cite{Ra11} for a proof
for the Atyiah-Singer operator under the chiral bag boundary condition, which can be easily modified for the Dirac Laplacian under the Dirichlet boundary condition.
\end{proof}
Note that if $V$ is the full exterior algebra bundle over $X$, the Dirichlet boundary condition is {\it not} the absolute boundary condition for differential
forms. Yet, they both generalize the Dirichlet boundary condition for functions.
\begin{theorem}\label{MVP}
Let $(V,\langle,\cdot,\cdot\rangle,\nabla)$ be a Dirac bundle over the compact Riemannian manifold $X$ with non-vanishing boundary $\partial X\neq\emptyset$, $Q$ the Dirac operator and $P$ the Dirac Laplacian. Then, any section
$\varphi\in C^{\infty}(X,V)$ satisfying $P\varphi=0$ can be written in terms of its values on the boundary as
\begin{equation}\label{rep}
\varphi(x)=-\left[\int_{\partial X}\langle \gamma(\nu)Q_yG(x,y)(\cdot),\varphi(y)\rangle_y\, d\text{vol}_{y\in \partial X}\right]^{\flat_x},
\end{equation}
where  $G$ is Green function of $P$ under the Dirichlet boundary condition, and $\flat:V^*\rightarrow V$ the bundle isomorphism induced by the Hermitian (Riemannian) structure $\<\cdot,\cdot>$.
The vector field $\nu$ denotes the inward pointing unit normal on $\partial X$.
\end{theorem}
\begin{proof}
By definition of Green function we have for any $\psi,\varphi\in C^{\infty}(X,V)$
\begin{equation}\label{int}
\int_X\langle [P_yG(x,y)]\psi(x),\varphi(y)\rangle\, d\text{vol}_{y\in X}=\int_X\langle \delta(y-x)\psi(x),\varphi(y)\rangle\, d\text{vol}_{y\in X}=\langle\psi(x),\varphi(x)\rangle_x.
\end{equation}
By partial integration the l.h.s. of (\ref{int}) becomes
\begin{equation}\label{int2}
\begin{split}
&\int_X\langle [P_yG(x,y)]\psi(x),\varphi(y)\rangle\, d\text{vol}_{y\in X}=\int_X\langle P_y[G(x,y)\psi(x)],\varphi(y)\rangle\, d\text{vol}_{y\in X}=\\
&\quad=\int_X\langle Q_yG(x,y)\psi(x),Q_y\varphi(y)\rangle\, d\text{vol}_{y\in X}-\int_{\partial X}\langle\gamma(\nu)[Q_yG(x,y)\psi(x),\varphi(y)\rangle\, d\text{vol}_{y\in \partial X}=\\
&\quad=\int_X\langle G(x,y)\psi(y),\underbrace{P_y\varphi(y)}_{=0}\rangle\, d\text{vol}_{y\in X}-\int_{\partial X}\langle\gamma(\nu)\underbrace{G(x,y)}_{=0}\psi(y),Q_y\varphi(y)\rangle\, d\text{vol}_{y\in \partial X}+\\
&\qquad-\int_{\partial X}\langle\gamma(\nu)Q_yG(x,y)\psi(x),\varphi(y)\rangle\, d\text{vol}_{y\in \partial X}.
\end{split}
\end{equation}
By comparing (\ref{int}) with (\ref{int2}) we obtain
\begin{equation}
\langle\psi(x),\varphi(x)\rangle_x=-\int_{\partial X}\langle\gamma(\nu)Q_yG(x,y)\psi(x),\varphi(y)\rangle\, d\text{vol}_{y\in \partial X},
\end{equation}
which is equivalent to (\ref{rep}).
\end{proof}
\noindent Theorem \ref{MVP} can be seen as a generalization of the mean value property for harmonic functions.
In this generality it appears to be a new result, as a literature search astoundingly shows.
\begin{proposition}
Let $E\in\mathcal{D}^{\prime}(X,V)$ a \textbf{fundamental solution} of the Dirac Laplacian $P$ for the Dirac bundle
$(V,\langle,\cdot,\cdot\rangle,\nabla)$ over the compact Riemannian manifold $X$ with non-vanishing boundary $\partial X\neq\emptyset$, that is
\begin{equation}
PE=\delta\mathbb{1}_{V}.
\end{equation}
For any $x\in X$ let $H^x\in\mathcal{D}^{\prime}(X,V)$ be the \textbf{corrector function}, that is the (distributional) solution to the boundary value problem
\begin{equation}
\left\{
  \begin{array}{l}
    P_yH^x(y)=0\quad (y\in X)\\
    H^x(y)=E(y-x)\quad (y\in\partial X).
  \end{array}
\right.
\end{equation}
Then, the Green function of $P$ under the Dirichlet boundary condition can be written as
\begin{equation}
G(x,y)=E(y-x)-H^x(y)
\end{equation}
for all $x,y\in X$.
\end{proposition}
\begin{proof}
By directly checking the definition of Green function we obtain for all $x,y\in X$
\begin{equation}
P_yG=P_yE(y-x)-P_yH^x(y)=\delta_x\mathbb{1}_{V_y},
\end{equation}
and for all $x\in X$, $y\in \partial X$
\begin{equation}
G(x,y)= E(y-x)-E(y-x)=0.
\end{equation}
The proof is completed.
\end{proof}
\noindent Theorem \ref{MVP} can be reformulated as follows.
\begin{corollary} Under the same assumptions as Theorem \ref{MVP}, we have for any $\psi,\varphi\in C^{\infty}(X,V)$
\begin{equation}
\int_X\langle \psi(x),\varphi(x)\rangle\, d\text{vol}_{x\in X}=-\int_{\partial X}\langle\gamma(\nu)Q_y\zeta[\psi](y),\varphi(y)\rangle\, d\text{vol}_{y\in \partial X},
\end{equation}
where
\begin{equation}\label{def_zeta}
\zeta[\psi](y):=\int_X G(x,y)\psi(x)\, d\text{vol}_{x\in X}.
\end{equation}
\end{corollary}

\begin{lemma}\label{bd_ind}
Let $P$ the Dirac Laplacian  for the Dirac bundle
$(V,\langle,\cdot,\cdot\rangle,\nabla)$ over the compact Riemannian manifold $X$ without boundary. Let us assume that
\begin{equation}
X=X_1\cup X_2,
\end{equation}
where $X_{1,2}$ are two $0$-codimensional Riemannian submanifolds of $X$ having disjoint interiors and the same boundary $\partial X_1=\partial X_2$. Then,
the Green functions $G^{X_1}$ and $G^{X_2}$ for the Dirac Laplacian $P$ on the Dirac bundle $V$ over $X_1$ and, respectively, $X_2$ define operators
$\zeta^{X_1}$ and $\zeta^{X2}$ via (\ref{def_zeta}), such that for any $\psi\in C^{\infty}(X,V)$ and any $y\in X$
\begin{equation}
\zeta^{X_1}[\psi](y)+\zeta^{X_2}[\psi](y)
\end{equation}
does not depend on $\partial X_{1,2}$.
\end{lemma}

\begin{proof}
For any $x\in X$ let $H^x\in\mathcal{D^{\prime}}(X,V)$ be the solution of
\begin{equation}
\left\{
  \begin{array}{l}
    P_yH^x(y)=0\quad (y\in X)\\
    H^x(y)=E(y-x)\quad (y\in\partial X_1=\partial X_2),
  \end{array}
\right.
\end{equation}
that is, the restriction of $H^x$to $X_{1,2}$ is the corrector function for the fundamental solution of the Dirac Laplacian on $X_{1,2}$.
Then, we obtain
\begin{equation}\label{int_zeta}
\begin{split}
&\zeta^{X_1}[\psi](y)+\zeta^{X_2}[\psi](y)=\\
&\quad=\int_{X_1}G^{X_1}(x,y)\psi(y)\, d\text{vol}_{x\in X_1}+\int_{X_2}G^{X_2}(x,y)\psi(y)\, d\text{vol}_{x\in X_2}=\\
&\quad=\int_{X}E(x,y)\psi(y)\, d\text{vol}_{x\in X}-\int_{X}H^{X,x}(y)\psi(y)\, d\text{vol}_{x\in X_1},
\end{split}
\end{equation}
which does not depend on $\partial X_{1,2}$: the first integral in the r.h.s of (\ref{int_zeta}) clearly is independent of $Y=\partial X_{1,2}$; the second integral is independent as well,
 because $H^{x}$ lies in the kernel of the elliptic operator $P$ over the boundaryless compact Riemannian manifold $X$
and is hence in $C^{\infty}(X,V)$, and, $Y$ is a zero measure subset of $X$.
\end{proof}

\par
The restriction of a Dirac bundle to a $1$-codimensional Riemannian
submanifold is again a Dirac bundle, as following theorem (cf.
\cite{Gi93} and \cite{Ba96}) shows.
\begin{theorem}\label{DiracBoundary}
Let $(V,\left<\cdot,\cdot\right>,\nabla,\gamma)$ be a Dirac bundle over the
Riemannian manifold $(X,g)$ and let $Y\subset X$ be a one
codimensional Riemannian submanifold with normal vector filed $\nu$. Then $(Y,
g|_Y)$ inherits a Dirac bundle structure by restriction. We mean by
this that the bundle $V|_Y$, the connection $\nabla|_{C^{\infty}(Y,
V|_Y)}$, the real algebra bundle homomorphism
$\gamma^Y:=-\gamma(\nu)\gamma|_{\Cl(Y,g|_Y)}$, and the Hermitian
(Riemannian) structure $\left<\cdot,\cdot\right>|_Y$ satisfy the defining
properties (iv)-(vi). The quadruple $(V|_Y, \left<\cdot,\cdot\right>|_Y,
\nabla|_{C^{\infty}(N, Y|_N)}, \gamma_Y)$ is called the
\textbf{Dirac bundle structure induced on $Y$} by the Dirac bundle
$(V,\left<\cdot,\cdot\right>,\nabla,\gamma)$ on $X$.
\end{theorem}

\section{Nash-Moser Generalized Inverse Function Theorem}
The generalization of the inverse function and implicit function theorems of calculus and the associated equation solution theorems have been pioneered by Nash
and Moser, who applied this technique to prove the Riemannian manifold embedding theorem (\cite{Na56})
and to solve small divisors problems \cite{Mo61, Mo61, Mo66}. Later, the technique was improved by H\"ormander (\cite{Ho76}) and Zehnder (\cite{Ze76}).
\begin{defi}
The family $(\mathcal{X}_s)_{s\ge0}$ is a \textbf{decreasing family of Banach spaces} if and only if $(\mathcal{X}_s,\|\cdot\|_s)$ is a Banach space for all $s\ge0$,
and for all $0\le s\le t$
\begin{equation}
\|x\|_s\le\|x\|_t\quad\text{ for all }x\in \mathcal{X}_t.
\end{equation}
We introduce the notation $\mathcal{X}_{\infty}:=\cap_{s\ge0}\mathcal{X}_s$
\end{defi}

\begin{defi}
Let $(\mathcal{X}_s)_{s\ge0}$ and $(\mathcal{Y}_s)_{s\ge0}$ be two families of decreasing Banach spaces. The map $\Phi:\mathcal{X}_s\rightarrow\mathcal{Y}_s$
satisfies the assumptions
\begin{itemize}
 \item $\textbf{(A1):}$ if and only if there exists a bounded open neighbourhood $U$ of $u_0\in\mathcal{X}_{s_0}$ for some $s_0\ge0$, such that for all $u\in U\cap \mathcal{X}_{\infty}$ the map $\Phi$ is
twice Fr\'echet-differentiable in $u$ and fulfills the tame estimate
 \begin{equation}
\|D^2\Phi(u).(v_1,v_2)\|_s\le C\left[\|v_1\|_{s+r} \|v_2\|_{s_0}+\|v_1\|_{s_0} \|v_2\|_{s+r}+\|v_1\|_{s_0} \|v_2\|_{s_0}\left(1+\|u-u_0\|_{s+t}\right)\right],
 \end{equation}
for all $s\ge0$, all $v_1,v_2\in\mathcal{X}_{\infty}$ and some fixed $r,t\ge0$. The constant $C>0$ is bounded for $s$ bounded.

 \item $\textbf{(A2):}$ if and only if there exists a bounded open neighbourhood $U$ of $u_0\in\mathcal{X}_{s_0}$ for some $s_0\ge0$, such that for all $u\in U\cap \mathcal{X}_{\infty}$
there exists a linear map $\Psi:\mathcal{Y}_{\infty}\rightarrow\mathcal{X}_{\infty}$ such that $D\Phi(u)\Psi(u)=\mathbb{1}$
and fulfills the tame estimate
 \begin{equation}
\|\Psi(u).v\|_s\le C\left[\|v\|_{s+p} +\|v\|_{s_0} \|u-u_0\|_{s+q}\right],
 \end{equation}
for all $s\ge0$, all $v\in\mathcal{X}_{\infty}$ and some fixed $p,q\ge0$. The constant $C>0$ remains bounded with $s$.
\end{itemize}
\end{defi}

\begin{defi}
The decreasing family of Banach spaces $(\mathcal{X}_s)_{s\ge0}$ satisfies
the \textbf{smoothing hypothesis} if there exists a family $(S_{\theta})_{\theta\ge1}$ of operators $S_{\theta}:\mathcal{X}_0\rightarrow \mathcal{X}_{\infty}$
 such that
\begin{equation}
\begin{split}
&\|S_{\theta}(u)\|_{\beta}\le C\theta^{(\beta-\alpha)_{+}}\|u\|_{\alpha}\quad (\alpha,\beta\ge0)\\
&\|S_{\theta}(u)-u\|_{\beta}\le C\theta^{\beta-\alpha}\|u\|_{\alpha}\quad (\alpha>\beta\ge0)\\
&\Big{\|}\frac{d}{d\theta}S_{\theta}(u)\Big{\|}_{\beta}\le C\theta^{\beta-\alpha}\|u\|_{\alpha}\quad (\alpha,\beta\ge0),
\end{split}
\end{equation}
where $\alpha_+:=\max\{a,0\}$. The constants
in the inequalities are uniform with respect to $\alpha,\beta$ when $\alpha,\beta$ belong to some
bounded interval.
\end{defi}

\begin{theorem}[\textbf{Nash-Moser}]\label{NMT}
Let $(\mathcal{X}_s)_{s\ge0}$ and $(\mathcal{Y}_s)_{s\ge0}$ be two families of decreasing Banach spaces each satisfying the smoothing hypothesis,
 and $\Phi:\mathcal{X}_s\rightarrow\mathcal{Y}_s$ satisfying assumptions $(A1)$ and $(A2)$. Let $s\ge s_0 +\max\{r,t\}+\max\{p,q\}$. Then:
\begin{itemize}
\item[(i)] There exists a constant $\varepsilon\in]0,1]$ such that, if $f\in\mathcal{Y}_{s+r+1}$ with
\begin{equation}\|f-\Phi(u_0)\|_{s+r+1}\le\varepsilon\end{equation}
the equation
\begin{equation}
\Phi(u)=f
\end{equation}
has a solution $u\in\mathcal{X}_s$ in the sense that there exists a sequence $(u_n)_{n\ge0}\subset\mathcal{X}_{\infty}$ such that for $n\rightarrow\infty$
\begin{equation}
u_n\rightarrow u\quad\text{ in } \mathcal{X}_s\quad\text{and}\quad \Phi(u_n)\rightarrow f\quad\text{ in } \mathcal{Y}_{s+p}
\end{equation}
\item[(ii)] If there exists $s^{\prime}>s$ such that $f\in\mathcal{Y}_{s^{\prime}+r+1}$, then the
solution constructed $u\in\mathcal{X}_{s^{\prime}}$.
\end{itemize}
\end{theorem}
\begin{proof}
See \cite{Be06} and \cite{Se16}.
\end{proof}

\begin{defi} For any $s\in\mathbf{R}$ the \textbf{Sobolev space} of complex valued functions over the Euclidean space is defined as
\begin{equation}
W^s(\mathbf{R}^n,\mathbf{C}^N):=\{u|\,(1+|x|^2)^{\frac{s}{2}}\widehat{u}(x)\in L^2(\mathbf{R}^n,\mathbf{C}^N)\},
\end{equation}
where $\widehat{\cdot}$ denotes the Fourier transform, and carries the scalar product and norm
\begin{equation}
\begin{split}
(u,v)_s&:=\left(1+|x|^2)^{\frac{s}{2}}\widehat{u}(x),(1+|x|^2)^{\frac{s}{2}}\widehat{v}(x)\right)_{L^2(\mathbf{R}^n,\mathbf{C}^N)}\\
\|u\|_s&:=\sqrt{(u,u)_s}.
\end{split}
\end{equation}
Let $s\in\mathbf{R}$. If $V$ is a complex or real vector bundle over the compact differentiable manifold $X$ \textbf{Sobolev space} of sections of $V$ over $X$ is denoted by
$W^s(X,V)$ and defined by local trivializations and a partition of unit of $X$.
\end{defi}
\begin{lemma}\label{lS}
For any $s\ge0$ the Sobolev space  $(W^s(\mathbf{R}^n,\mathbf{C}^N),(\cdot,\cdot)_s)$  is a Hilbert space and a
Banach space. There exists a constant $c_s>0$ and a $s_0$ with $0\le s_0<s$ such that
\begin{equation}\label{inS}
\|uv\|_s\le c_s\left(\|u\|_s\|v\|_{s_0}+\|u\|_{s_0}\|v\|_{s}\right).
\end{equation}
Moreover,
\begin{equation}\label{inS2}
\|uv\|_s\le 2c_s\|u\|_s\|v\|_s.
\end{equation}
\end{lemma}
\begin{proof}
We just prove inequality (\ref{inS}), because the competeness result is standard in functional analysis. We assume first that $s$ is a non negative integer.
For any $s_0=0,\dots,s-1$ we have for any $\alpha\in\mathbf{N}^n$ such that $|\alpha|\le s$
\begin{equation}
\partial^{\alpha}(uv)=\sum_{\beta\le\alpha}\binom{\alpha}{\beta} (\partial^{\alpha-\beta}u)(\partial^{\beta}v)
=\left[\sum_{\substack{\beta\le\alpha\\|\beta|\le s_0}}+\sum_{\substack{\beta\le\alpha\\|\beta|> s_0}}\right]\binom{\alpha}{\beta}(\partial^{\alpha-\beta}u)(\partial^{\beta}v),
\end{equation}
 from which (\ref{inS}) follows. The general case for a real $s\ge0$ is proved by norm interpolation (cf. \cite{Tr77}).\\
\end{proof}
We can now prove a technical Lemma which will be essential in the proof the of the Hodge conjecture in the next section, showing
the existence of two differentiable submanifolds of a projective manifold $X$ without boundary
satisfying a certain PDE under a certain injectivity assumption.\par
The generic set up is given by the atlas $(U_i,\Phi_i)_{i=0,\dots K}$ for $X$ as in Proposition \ref{cext}
and two differentiable submanifolds
$Y_{1}\subset Y_0\subset X$ of real codimension $1$ and $2$ as in Definition \ref{sub},
and two differentiable submanifolds with boundary $B_{0,1}\subset X$, such that $\partial B_{0,1}=Y_{0,1}$, given by
\begin{equation}
\begin{split}
(U_i\cap X,\Pi_{2n}\circ\varphi_i\circ\Phi_i)_{i=0,\dots,K}:\text{ Atlas for }X\\
(U_i\cap Y_0,\Pi_{2n-1}\circ\varphi_i\circ\Phi_i)_{i=0,\dots,K}:\text{ Atlas for }Y_0\\
(U_i\cap Y_1,\Pi_{2n-2}\circ\varphi_i\circ\Phi_i)_{i=0,\dots,K}:\text{ Atlas for }Y_1\\
(U_i\cap B_0,\Pi_{2n}^+\circ\varphi_i\circ\Phi_i)_{i=0,\dots,K}:\text{ Atlas for }B_0\\
(U_i\cap B_1,\Pi_{2n-1}^+\circ\varphi_i\circ\Phi_i)_{i=0,\dots,K}:\text{ Atlas for }B_1,
\end{split}
\end{equation}
where

\begin{equation}
\begin{array}{ll}
\Pi_k:\mathbf{R}^{2n}\rightarrow \mathbf{R}^k\qquad\qquad &\Pi_k^+:\mathbf{R}^{2n}\rightarrow \mathbf{R}^k\times[0,+\infty[\\
(a_1,\dots,a_{2n})\mapsto (a_1,\dots,a_{k})\qquad\qquad &(a_1,\dots,a_{2n})\mapsto (a_1,\dots,a_k,1_{[0,+\infty[}(a_{k+1}))
\end{array}
\end{equation}
denote the projections of $\mathbf{R}^{2n}$ onto the subspace $\mathbf{R}^{k}$ and the half-space  $\mathbf{R}^k\times[0,+\infty[$,
and $\varphi_i:\mathbf{R}^{2n}\hookrightarrow \mathbf{R}^{2n}$are local diffeomorphisms for all $i$.
In order for $Y_{0,1}$ and $B_{0,1}$ to be well defined we need $(\varphi_i)_{i=0,\dots,N}$ to fulfill the assumptions of Proposition \ref{propdefsub}, the compatibility conditions.
\begin{lemma}\label{TL}
Let $X$ be a  compact complex manifold without boundary and $n$ be the complex dimension of $X$.
For any $\varphi=(\varphi_0,\dots,\varphi_K)$ defining $0$ and, respectively, $1$ codimensional complex submanifolds $B_{0,1}\subset X$
as in Definition \ref{sub} and here above let
\begin{equation}
\Xi^{\omega}(\varphi):= i_{\partial B_1}^*\left(\gamma^{\partial B_0}(\nu^{\partial B_1})Q^{\partial B_0}\zeta^{\partial B_1}\gamma^{X}(\nu^{\partial B_0})Q^{X}\zeta^{\partial B_0}\omega\right)-\mu_{\partial B_1}.
\end{equation}
Then,
\begin{itemize}
\item[(i)]  For the atlas $(U_i,\Phi_i)_{i=0,\dots,K}$ there exist a differential form-valued function
 $F^{\omega}=F^{\omega}(\gamma,\Gamma)$ for $\gamma\in \bigoplus_{i=0}^K \mathbf{R}^{2n}$ and
$\Gamma\in\bigoplus_{i=0}^K \mathbf{R}^{2n\times 2n}$ such that for all $x\in X$
\begin{equation}
\Xi^{\omega}(\varphi)(x)=F^{\omega}(\varphi(x),T\varphi(x))\text{ where } \varphi:=(\varphi_0,\varphi_1,\dots,\varphi_K).
\end{equation}
with the convention that $\varphi_j|_{\Phi_i(U_i)}:\equiv0$ for $j\neq i$. Moreover, $F^{\omega}$ is an affine functional of $\omega$.
\item[(ii)]
For any given $\omega\in\Omega^{n-1,n-1}(X,\mathbf{C})$ let us define
\begin{equation}\label{omega_set}
\begin{split}
&\mathcal{F}^{\omega,\{U_i\}_{i=0,\dots,K}}:=\left\{\psi:=(\psi_0,\psi_1,\dots,\psi_K)|\;\psi_i:\Phi_i(U_i)\rightarrow\psi_i(\Phi_i(U_i))\text{ is a } \right.\\
&\qquad\qquad\qquad\qquad\left.\text{diffeomorphism for all }i=0,\dots,K\text{ and }DF^{\omega}(\psi,T\psi)\text{ is injective}\right.\}.
\end{split}
\end{equation}
If $\mathcal{F}^{\omega,\{U_i\}_{i=0,\dots,K}}\neq\emptyset$, there exist local diffeomorphisms  $(\varphi_i)_{i=0,\dots,N}$  defining
  an oriented $0$-codimensional differentiable submanifold of $X$, $B_0$ with boundary $\partial B_0$, and an oriented $0$-codimensional differentiable submanifold  of $\partial B_0$,
$B_1$ with boundary $\partial B_1$, such that the equality
\begin{equation}\label{eq18}
i_{\partial B_1}^*\left(\gamma^{\partial B_0}(\nu^{\partial B_1})Q^{\partial B_0}\zeta^{\partial B_1}\gamma^{X}(\nu^{\partial B_0})Q^{X}\zeta^{\partial B_0}\omega\right)=\mu_{\partial B_1},
\end{equation}
holds true,
\end{itemize}
where we have utilized:
\begin{itemize}
\item The antiholomorphic bundle on $X$ is a Dirac bundle by Proposition \ref{prop21} and is denoted by $(V,\gamma, \langle\cdot,\cdot\rangle,\nabla)$ with
corresponding Dirac operator, the Dirac-Dolbeault operator $Q^X$.
\item The operator $Q^{\partial B_0}$ is the Dirac operator on $\partial B_0$ corresponding to the Dirac bundle structure induced by Theorem \ref{DiracBoundary} by the Dirac bundle
structure on $X$, and $\nu^{\partial B_0}$ the inward pointing unit normal vector field to $\partial B_0$ in $X$.
\item The operator $Q^{\partial B_1}$ is the Dirac operator on $\partial B_1$ corresponding to the Dirac bundle structure induced by Theorem \ref{DiracBoundary} by the Dirac bundle
structure on $\partial B_0$, and $\nu^{\partial B_1}$ the inward pointing unit normal vector field to $\partial B_1$ in $\partial B_0$,
\item the  $(n-1, n-1)$-differential form $\mu_{\partial B_1}$ is the volume form on $\partial B_1$.
\item the Green functions for the Hodge-Kodaira Laplacians on $X\setminus B_0$ and $B_0$, and, respectively $B_0\setminus B_1$ and $B_1$  as in Proposition \ref{prop31} are denoted by $G^{X\setminus B_0}$ and $G^{B_0}$,
 and, respectively, by $G^{B_0\setminus B_1}$ and $G^{B_1}$.
\item The operators
\begin{equation}
\begin{split}
\zeta^{\partial B_0}[\omega](y)&:=\left(\int_{X\setminus B_0}G^{X\setminus B_0}(x,y)\omega(x)+\int_{B_0}G^{B_0}(x,y)\omega(y)\right)d\text{vol}_{x\in X}\\
\zeta^{\partial B_1}[\eta](y)&:=\left(\int_{\partial B_0\setminus B_1}G^{B_0\setminus B_1}(x,y)\eta(x)+\int_{B_1}G^{B_1}(x,y)\eta(x)\right)d\text{vol}_{x\in \partial B_0}\\
\end{split}
\end{equation}
are complex bundle homomorphisms on $\Lambda(TX^{0,1})^*|_{\partial B_0}$ and $\Lambda(TB_0^{0,1})^*|_{\partial B_1}$, respectively.
\end{itemize}
Moreover, $\partial B_1$ is a \textbf{complex hypersurface} of $X$.
\end{lemma}

\noindent Before proving this Lemma we need to introduce some required additional results following \cite{DjOk13}.
\begin{defi}\label{defLambda} Let $X$ be a $(n+2)$-real dimensional complex manifold, $J$ its natural almost complex structure, $g^X$ its hermitian metric,
and $Y$ an $n$-real dimensional real submanifold of $X$, with immersion $i_Y:Y\rightarrow X$. Note that $Ti_Y=i_Y$.
 The tangent bundle $TY$ is identified with a subbundle of $TX$. The Riemaniann metric on $Y$ induced by $g^X$ reads
\begin{equation}
g^Y(A,B):=g^X(i_YA,i_YB)
\end{equation}
for all $A,B\in TY$.\\
Let $\xi_1$ and $\xi_2$ be two mutually orthogonal unit normals to $TY$.
The \textbf{almost contact tensor} $F\in\text{Hom}(TY)$ is defined via the equation
\begin{equation}
Ji_Y V=i_Y FV+u^1(V)\xi_1+u^2(V)\xi_2\quad(V\in TY),
\end{equation}
where $u^1,u^2$ are real differential 1-forms on Y.
\end{defi}
\begin{lemma}\label{lemma_lambda_calc}
For local tangent fields $U_1,U_2$ in $TY$ we have
\begin{equation}
J\xi_a=-i_Y U_a+\lambda_{a,1}\xi_1+\lambda_{a,2}\xi_2,
\end{equation}
for appropriate $\lambda_{a,b}$. Then,
\begin{equation}\label{def_lambda}
\begin{split}
\lambda_{1,1}&=g^X(J\xi_1,\xi_1)=0\\
\lambda_{2,2}&=g^X(J\xi_2,\xi_2)=0\\
\lambda_{1,2}&=-\lambda_{2,1}=:\lambda=g^X(J\xi_1,\xi_2).
\end{split}
\end{equation}
Moreover, the almost contact tensor is antisymmetric, and
\begin{equation}
\begin{split}
FU_1&=-\lambda U_2\\
FU_2&=+\lambda U_1.
\end{split}
\end{equation}
\end{lemma}
\begin{proof}
See \cite{DjOk13}.
\end{proof}
\begin{proposition}[\textbf{Djori\'{c}, Okumura}]\label{prop_lambda}
 Let $Y$ be a real submanifold of codimension $2$ of a complex manifold $X$ and let $\lambda$ be the function defined by (\ref{def_lambda}).
Then:
\begin{itemize}
\item $Y$ is a complex hypersurface if and only if $\lambda^2(y)=1$ for any $y\in Y$.
\item $Y$ is a CR submanifold of CR dimension $\frac{n-2}{2}$ if $\lambda(y)=0$ for any $y\in Y$.
\end{itemize}
\end{proposition}
\begin{proof}
See \cite{DjOk13}.
\end{proof}

\noindent We can now proceed with the
\begin{proof}[\text{Proof of Lemma \ref{TL}}]\text{}\\
(i): Any $\varphi$ in $(\varphi_i)_{i=0,\dots,K}$ is contained in the defining expression for $\Xi^{\omega}(\varphi)$ by means of the vector field tangential to the coordinate lines
$\frac{\partial}{\partial z_1},\dots, \frac{\partial}{\partial z_n},\dots,\frac{\partial}{\partial \overline{z}_1},\dots, \frac{\partial}{\partial \overline{z}_n}$
and the corresponding o.n. system defined by means of the Gram-Schmidt orthogonalization procedure $e_1,\dots,e_n,\overline{e}_1,\dots,\overline{e}_n$. More exactly:
\begin{itemize}
\item The o.n. frames are $C^{\infty}$ functions of $(\varphi, D\varphi)$:
\begin{equation}
\begin{split}
  &\{\epsilon_1,\dots,\epsilon_{2n}\}:=\{e_1,\dots,e_n,\overline{e}_1,\dots,\overline{e}_n\} \text{ for } TX|_U\\
  &\{\epsilon_1,\dots,\epsilon_{2n-1}\}:=\{e_1,\dots,e_n,\overline{e}_1,\dots,\overline{e}_{n-1}\} \text{ for } T\partial B_0|_U\\
  &\{\epsilon_1,\dots,\epsilon_{2n-2}\}:=\{e_1,\dots,e_{n-1},\overline{e}_1,\dots,\overline{e}_{n-1}\} \text{ for } T\partial B_1|_U.
\end{split}
\end{equation}
\item The real algebra bundle homomorphisms are $C^{\infty}$ functionals of $(\varphi, T\varphi)$:
\begin{equation}
\begin{split}
  \gamma^X(v)&=\sqrt{2}(\exte(v^{1,0})-\inte(v^{0,1}))\text{ for }v=v^{1,0}\oplus v^{0,1}\in TX\\
  \gamma^{\partial B_0}(v)&=-\gamma^X(\nu^{\partial B_0})\gamma^X(v)\text{ for }v\in T\partial B_0\\
  \gamma^{\partial B_1}(v)&=-\gamma^{\partial B_0}(\nu^{\partial B_1})\gamma^{\partial B_0}(v)\text{ for }v\in T\partial B_1.
\end{split}
\end{equation}
We mean by this that these functionals are Fr\'{e}chet-differentiable an infinite number of times.
\item The lifts of the Levi-Civita connections are $C^{\infty}$ functionals of $(\varphi, T\varphi)$:
\begin{align}
  \nabla^X&=d^X+\omega^X,\text{ where }&\omega^X=\frac{1}{4}\sum_{i,j,k}{\prescript{X}{}{\Gamma}}_{k,i}^j\gamma^X(\epsilon_j)\gamma^X(\epsilon_k)(\epsilon_i)^\flat\\
  \nabla^{\partial B_0}&=d^{\partial B_0}+\omega^{\partial B_0},\text{ where }&\omega^{\partial B_0}=\frac{1}{4}\sum_{i,j,k}{\prescript{\partial B_0}{}{\Gamma}}_{k,i}^j\gamma^{\partial B_0}\epsilon_j\gamma^{\partial B_0}(\epsilon_k)(\epsilon_i)^\flat\\
  \nabla^{\partial B_1}&=d^{\partial B_1}+\omega^{\partial B_1},\text{ where }&\omega^{\partial B_1}=\frac{1}{4}\sum_{i,j,k}{\prescript{\partial B_1}{}{\Gamma}}_{k,i}^j\gamma^{\partial B_1}\epsilon_j\gamma^{\partial B_1}(\epsilon_k)(\epsilon_i)^\flat
\end{align}
where $\omega^X, \omega^{\partial B_0}, \omega^{\partial B_1}$ are the local connection homorphisms depending on the Christoffel symbols
${\prescript{X}{}{\Gamma}}_{k,i}^j, {\prescript{\partial B_0}{}{\Gamma}}_{k,i}^j, {\prescript{\partial B_1}{}{\Gamma}}_{k,i}^j$,
which  are a functional of the first derivatives of the Riemannian metrics and its inverse.
(see\cite{BoWo93}, page 15).
\item The Dirac operators are $C^{\infty}$ functional of $(\varphi, T\varphi)$:
\begin{equation}
\begin{split}
  Q^X&=\sum_{i=1}^n\gamma^X(e_i)\nabla^X_{e_i}+\sum_{i=1}^n\gamma^X(\bar{e}_i)\nabla^X_{\bar{e}_i}\\
  Q^{\partial B_0}&=\sum_{i=1}^n\gamma^{\partial B_0}(e_i)\nabla^{\partial B_0}_{e_i}+\sum_{i=1}^{n-1}\gamma^{\partial B_0}(\bar{e}_i)\nabla^{\partial B_0}_{\bar{e}_i}\\
  Q^{\partial B_1}&=\sum_{i=1}^{n-1}\gamma^{\partial B_1}(e_i)\nabla^{\partial B_1}_{e_i}+\sum_{i=1}^{n-1}\gamma^{\partial B_1}(\bar{e}_i)\nabla^{\partial B_1}_{\bar{e}_i}.
\end{split}
\end{equation}
\item The Green functions for the Hodge-Kodaira Laplacians are $C^{\infty}$ functionals of $(\varphi, T\varphi)$:
\begin{equation}
\begin{split}
  G^{B_0}=G^{B_0}(x,y)\text{ for } y\in \partial B_0 \text{ and } x\in X\\
  G^{X\setminus B_0}=G^{X\setminus B_0}(x,y)\text{ for } y\in \partial (X\setminus B_0) \text{ and }x\in X\\
  G^{B_1}=G^{B_1}(x,y)\text{ for } y\in \partial B_1 \text{ and } x\in \partial B_0\\
  G^{\partial B_0\setminus B_1}=G^{\partial B_0\setminus B_1}(x,y)\text{ for } y\in \partial(B_0\setminus B_1) \text{ and }x\in \partial B_0.
\end{split}
\end{equation}
\item The bundle homorphisms are a priori $C^{\infty}$ functionals of $(\varphi, T\varphi)$:
\begin{equation}
\begin{split}
\zeta^{\partial B_0}=\zeta^{\partial B_0}(y)\text{ for }y\in \partial B_0\\
\zeta^{\partial B_1}=\zeta^{\partial B_1}(y)\text{ for }y\in \partial B_1.
\end{split}
\end{equation}
However, by Lemma \ref{bd_ind}, they only depend on $y$ and hence on $\varphi$.
\item The volume form is a $C^{\infty}$ functional of $(\varphi, T\varphi)$:
\begin{equation}
\mu_{\partial B_1}=\det h|_{\partial B_1}\left(\left\{\frac{\partial}{\partial z_1},\dots,
\frac{\partial}{\partial z_{n-1}},\frac{\partial}{\partial \bar{z}_1},\dots,\frac{\partial}{\partial \bar{z}_{n-1}}\right\}\right)^{\frac{1}{2}}
dz_1\wedge\dots\wedge dz_{n-1}\wedge d\bar{z}_1\wedge\dots\wedge d\bar{z}_{n-1}
\end{equation}
where $h|_{\partial B_1}$ is the Riemannian metric on $\partial B_1$, the restriction of the Fubini-Study metric on $\mathbf{C}P^N\supset X$.
\end{itemize}
We conclude that for the atlas $(U_i,\Phi_i)_{i=0,\dots,K}$ there exist a differential form-valued function
 $F^{\omega}=F^{\omega}(\gamma,\Gamma)$ for $\gamma\in \bigoplus_{i=0}^K \mathbf{R}^{2n}$ and $\Gamma\in\bigoplus_{i=0}^K \mathbf{R}^{2n\times 2n}$ such that
\begin{equation}
\Xi^{\omega}(\varphi)(x)=F^{\omega}(\varphi(x),T\varphi(x))\text{ where } \varphi:=(\varphi_0,\varphi_1,\dots,\varphi_K).
\end{equation}
with the convention that $\varphi_j|_{\Phi_i(U_i)}:\equiv0$ for $j\neq i$.
Note that $F^{\omega}$ is an affine functional of $\omega$. \\
(ii): For any $\varphi=(\varphi_0,\dots,\varphi_K)$ we have to solve the equation
\begin{equation}\label{eqU}
\Xi^{\omega}(\varphi)=0,
\end{equation}
while making sure that the diffeomorphisms $(\varphi_i)_{i=0,\dots,K}$ satisfy the compatibility condition.
Moreover, we have to choose $\varphi$ such that $\partial B_1$ is a complex hypersurface of $X$.
Following Proposition \ref{prop_lambda}, we have to solve the equation
\begin{equation}
\lambda^2(\varphi_i)-1=0,
\end{equation}
for all $i=0,\dots,K$, where we insert the expression for $\lambda$ computed in Lemma \ref{lemma_lambda_calc}, that is
\begin{equation}
\begin{split}
\lambda(\varphi_i)&=g^{\text{FS}}(J\xi_1,\xi_2)\\
\xi_1 &= (T(\varphi_i\circ\Phi_i)^{n-1})^{\flat}\\
\xi_2 &= (T(\varphi_i\circ\Phi_i)^{n})^{\flat}.
\end{split}
\end{equation}

\noindent We proceed now to verify the fulfillment of the Nash-Moser inverse function theorem, making sure that our construction of the differentiable manifolds
$Y_{0,1}$ and $B_{0,1}$ is well defined.
\begin{itemize}
\item \textbf{Functional between Banach spaces: } We consider the two family decreasing Banach spaces given by
\begin{equation}
\begin{split}
&\mathcal{X}_s:=\bigoplus_{i=0}^K W^s\left(\Phi_i(U_i),\mathbf{R}^{2n}\right)\\
&\|\varphi)\|_s^2:=\sum_{i=0}^K\|\varphi_i\|_s^2\\
&\\
&\mathcal{Y}_s:=\bigoplus_{i=0}^K W^s\left(U_i,\Lambda(TX^*)\otimes\mathbf{C}\right)\bigoplus_{0\le i<j\le K} W^s\left(\Phi_i(U_i)\cap\Phi_j(U_j),\mathbf{R}^{2n}\right)\bigoplus_{i=0}^K W^s\left(U_i,\mathbf{R}\right)\\
&\|(\zeta,\nu,\lambda)\|_s^2:=\sum_{i=0}^K\|\zeta_i\|_s^2+\sum_{0\le i<j\le K}\|\nu_{i,j}\|_s^2+\sum_{i=0}^K\|\lambda_i\|_s^2,
\end{split}
\end{equation}
with the corresponding Sobolev norm $\|\cdot\|_s$ defined for any $s\ge0$ on the appropriate spaces. The families $(\mathcal{X}_s, \|\cdot\|_s)_{s\ge0}$
and $(\mathcal{Y}_s, \|\cdot\|_s)_{s\ge0}$ satisfy the smoothing hypothesis (cf. \cite{Ra89} page 25).  Note that, by the Sobolev embedding theorem,
\begin{equation}
\mathcal{X}_{\infty}=\bigcap_{s\ge0}\bigoplus_{i=0}^K W^s\left(\Phi_i(U_i),\mathbf{R}^{2n}\right)
=\bigoplus_{i=0}^K C^{\infty}\left(\overline{\Phi_i}(\overline{U_i}),\mathbf{R}^{2n}\right).
\end{equation}
Equation (\ref{eqU}) for the two charts $i,j$ and the compatibility of the definitions on the intersection of $\Phi_i(U_i)$ with $\Phi_j(U_j)$ can be expressed as
\begin{equation}
\Theta(\varphi):=(\underbrace{\Xi^{\omega}(\varphi)}_{=:\Theta_1(\varphi)},\underbrace{(\varphi_i^{-1}-(\Phi_i\circ\Phi_j^{-1})\circ\varphi_j^{-1})_{i<j}}_{=:\Theta_2(\varphi)}, \underbrace{(\lambda^2(\varphi_i)-1)_i)}_{=:\Theta_3(\varphi)})=0\in Y_{\infty}.
\end{equation}
Throughout the rest of this proof we will repeatedly make use of the fact that $i^*_{\partial B_1}(\eta)(x)=\eta(x)$ for $x\in\partial B_1$,
 because $i_{\partial B_1}(x)=x$ and $Ti_{\partial B_1}(x)=\mathbb{1}_{T_x\partial B_1}$.
 \item $\textbf{Assumption (A1):}$ Let $S_1\subset\mathcal{F}^{\omega, \{U_i\}_{i=0,\dots}}\neq\emptyset$ be a bounded open set of $\mathcal{X}_{0}$.
Such a $S_1$ exists by continuity.
 For all $\varphi\in R_1\cap \mathcal{X}_{\infty}$
 the map $\Xi^{\omega}(\varphi)$ is twice Fr\'echet-differentiable in $\varphi$, and for any  $s\ge0$,
and $v_1,v_2\in\mathcal{X}_{\infty}$ we have
\begin{equation}
\begin{split}
&D^2\Theta_1(\varphi).(v_1,v_2)=\\
&=\left.\frac{\partial}{\partial t_1}\right|_{t1:=0}\left.\frac{\partial}{\partial t_2}\right|_{t2:=0}\Xi^{\omega}(\varphi+t_1v_1+t_2v_2)=\\
&=D_{.11} F^{\omega}(\varphi,T\varphi).(v_1,v_2)
+D_{.22} F^{\omega}(\varphi,T\varphi).(Tv_1,Tv_2)+\\
&+D_{.12} F^{\omega}(\varphi,D\varphi).(v_1,Tv_2)+D_{.21} F^{\omega}(\varphi,T\varphi).(Tv_1,v_2)\in\Lambda((TX^{0,1})^*),
\end{split}
\end{equation}
and, by Lemma \ref{lS} for any  $s\ge0$, there exists a constant $c_s>0$ and a $s_0$ with $0\le s_0<s$ such that
\begin{equation}
\begin{split}
&\|D^2\Theta_1(\varphi).(v_1,v_2)\|_s\le\\
&\quad\le\sum_{i,j}\left\{c_s\left[\|D_{.11}F^{\omega}(\varphi,T\varphi)\|_s\|v_1^iv_2^j\|_{s_0}+\|D_{.11}F^{\omega}(\varphi,T\varphi)\|_{s_0}\|v_1 ^iv_2^j\|_s\right]\right.+\\
 &\quad+\sum_{h}C\left[\|D_{.12}F^{\omega}(\varphi,T\varphi)\|_s\|v_1^iTv_2^{j,h}\|_{s_0}+\|D_{.12}F^{\omega}(\varphi,T\varphi)\|_{s_0}\|v_1 ^iTv_2^{j,h}\|_s\right]+\\
 &\quad+\sum_{h,l}\left.C\left[\|D_{.22}F^{\omega}(\varphi,T\varphi)\|_s\|Tv_1^{i,l}Tv_2^{j,h}\|_{s_0}+\|D_{.22}F^{\omega}(\varphi,T\varphi)\|_{s_0}\|Tv_1^{i,l}Tv_2^{j,h}\|_s\right]\right\}\le\\
 &\le C_s^1(\varphi)\left[\|v_1\|_{s+1}\|v_2\|_{s_0+1}+\|v_1\|_{s_0+1}\|v_2\|_{s+1}\right],
\end{split}
\end{equation}
where $C_s^1(\varphi)\le C_s^1(S_1)$, which remains bounded when $s$ remains bounded. Hence, (A1) is fulfilled for $\Theta_1(\varphi)$ and $\varphi\in S_1\cap X_{\infty}$.\\
Let $S_2\neq\emptyset$ be a bounded open set of $\mathcal{X}_{0}$ such that $S_1\cap S_2\neq\emptyset$.
For the second component of $\Theta$ we have, for all $\varphi\in S_2\cap \mathcal{X}_{\infty}$
\begin{equation}
\begin{split}
D^2\Theta_2(\varphi).(v_1,v_2)&=\left.\frac{\partial}{\partial t_1}\right|_{t1:=0}\left.\frac{\partial}{\partial t_2}\right|_{t2:=0}\Theta_2(\varphi+t_1v_1+t_2v_2)=\\
&=\left(-[T\varphi_i(\varphi_i^{-1})]^{-1}T^2\varphi_i(\varphi_i^{-1}).(T\varphi_i(\varphi_i^{-1}).v_1^i,v_2^i)+\right.\\
&\qquad-T^2A_{i,j}(\varphi_j^{-1})([T\varphi_j(\varphi_j^{-1})]^2v_1^i,v_2^j)+\\
&\qquad\left.TA_{i,j}(\varphi_j^{-1}))[T\varphi_j(\varphi_j^{-1})]^{-2}T^2\varphi_j(\varphi_j^{-1}).(T\varphi_j(\varphi_j^{-1}).v_1^j,v_2^j)\right)_{i<j},
\end{split}
\end{equation}
where $A_{i,j}(\cdot):=\Phi_i\circ\Phi_j^{-1}(\cdot)$ and $T^2\varphi_i$, $T^2\varphi_j$, $T^2A_{i,j}$ are vector valued bilinear forms
By Proposition \ref{cext} the domain of definition of $A_{i,j}$ and the images of $\varphi_i^{-1}$ and $\varphi_j^{-1}$ are contained in a compact set.
Therefore, the Sobolev norms of the linear operator $\|TA_{i,j}(\varphi_j^{-1})\|_s$ and of the bilinear operator $\|T^2A_{i,j}(\varphi_j^{-1})\|_s$, as well
as those of $\|[T\varphi_i(\varphi_i^{-1})]^{-1}\|_s$, $\|T\varphi_i(\varphi_i^{-1})\|_s$ and $\|T^2\varphi_i(\varphi_i^{-1})\|_s$
 remain bounded for all $\varphi\in S_2\cap \mathcal{X}_{\infty}$.
By Lemma \ref{lS} for any  $s\ge0$, there exists a $C_s^2(\varphi)\le C_s^2(S_2)$, which remains bounded when $s$ remains bounded,
 and a $s_0$ with $0\le s_0<s$ such that
\begin{equation}
\|D^2\Theta_2(\varphi).(v_1,v_2)\|_s\le C_s^2(\varphi)\left[\|v_1\|_{s+1}\|v_2\|_{s_0+1}+\|v_1\|_{s_0+1}\|v_2\|_{s+1}\right],
\end{equation}
Hence, (A1) is fulfilled for $\Theta_2(\varphi)$ and $\varphi\in S_2\cap X_{\infty}$.\\
Let $S_3\neq\emptyset$ be a bounded open set of $\mathcal{X}_{0}$ such that $S_1\cap S_2\cap S_3\neq\emptyset$. For the third component we have
\begin{equation}
D^2\Theta_3(\varphi).(v_1,v_2)=\left((D\lambda(\varphi_i).v_1)(D\lambda(\varphi_i).v_2)+\lambda(\varphi_i)D^2\lambda(\varphi_i).(v_1,v_2)\right)_{i<j},
\end{equation}
where
\begin{equation}
\begin{split}
\lambda(\varphi_i)&=g^{\text{FS}}\left(J((T\varphi_i)^{n-1}(\Phi_i)T\Phi_i)^{\flat}, ((T\varphi_i)^n(\Phi_i)T\Phi_i)^{\flat}\right)\\
D\lambda(\varphi_i).v_1&=g^{\text{FS}}\left(J((Tv_1)^{n-1}(\Phi_i)T\Phi_i)^{\flat}, ((T\varphi_i)^n(\Phi_i)T\Phi_i)^{\flat}\right)+\\
&\quad+g^{\text{FS}}\left(J((T\varphi_i)^{n-1}(\Phi_i)T\Phi_i)^{\flat}, ((Tv_1)^n(\Phi_i)T\Phi_i)^{\flat}\right).\\
D^2\lambda(\varphi_i).(v_1,v_2)&=g^{\text{FS}}\left(J((Tv_1)^{n-1}(\Phi_i)T\Phi_i)^{\flat}, ((Tv_2)^n(\Phi_i)T\Phi_i)^{\flat}\right)+\\
&\quad+g^{\text{FS}}\left(J((Tv_2)^{n-1}(\Phi_i)T\Phi_i)^{\flat}, ((Tv_1)^n(\Phi_i)T\Phi_i)^{\flat}\right).
\end{split}
\end{equation}
By Lemma \ref{lS} for any  $s\ge0$ there exists a constant $c_s>0$ and a $s_0$ with $0\le s_0<s$ such that,
\begin{equation}
\begin{split}
&\|(D\lambda(\varphi_i).v_1)(D\lambda(\varphi_i).v_2)+\lambda(\varphi_i)D^2\lambda(\varphi_i).(v_1,v_2)\|_s\le \\
&\le c_s^2\left[\|g^{\text{FS}}\|_s\|T\varphi_i\|_{s_0}\|Tv_1\|_{s_0}+\|g^{\text{FS}}\|_{s_0}\|T\varphi_i\|_{s}\|Tv_1\|_{s}\right]\cdot\\
&\quad\cdot\left[\|g^{\text{FS}}\|_s\|T\varphi_i\|_{s_0}\|Tv_2\|_{s_0}+\|g^{\text{FS}}\|_{s_0}\|T\varphi_i\|_{s}\|Tv_2\|_{s}\right]+\\
&\quad+ 4c_s\|g^{\text{FS}}\|_s\|T\varphi_i\|_{s}^2\left[\|g^{\text{FS}}\|_s\|Tv_1\|_{s_0}\|Tv_2\|_{s_0}+\|g^{\text{FS}}\|_{s_0}\|Tv_1\|_{s}\|Tv_2\|_{s}\right].
\end{split}
\end{equation}
By analyzing the expression (\ref{FS}) we see that all Sobolev norms of the Fubini-Study metric are bounded, that is
\begin{equation}\label{FS_Sob}
\|g^{\text{FS}}\|_{s}<+\infty\text{ for all }s\ge0.
\end{equation}
Hence, we infer the existence of  a $C_s^3(\varphi)\le C_s^3(S_3)$, a  positive functional of $\varphi$
which remains bounded with $s$, such that
\begin{equation}
\begin{split}
|(D\Theta_3(\varphi_i).(v_1,v_2)\| \le C_s^3(\varphi_i)\left[\|v_1\|_{s_0+1}|v_2\|_{s+1}+\|v_1\|_{s+1}|v_2\|_{s_0+1}+\|v_1\|_{s+1}|v_2\|_{s+1}\right].
\end{split}
\end{equation}
Hence, (A1) is fulfilled for $\Theta_3(\varphi)$ and $\varphi\in S_3\cap X_{\infty}$. We conclude that
(A1) is fulfilled for $\Theta(\varphi)$ and $\varphi\in S\cap X_{\infty}$ for $S:=S_1\cap S_2\cap S_3\neq\emptyset$ and $C_s(\varphi):=\max_{i=1,2,3}C_s^i(\varphi)$.
 \item $\textbf{Assumption (A2):}$ Let $R_1\subset\mathcal{F}^{\omega, \{U_i\}_{i=0,\dots}}\neq\emptyset$ be a bounded open set of $\mathcal{X}_{0}$.
Such a $R_1$ exists by continuity. The first Fr\'echet-derivative of $\Theta_1(\varphi)$ reads
\begin{equation}
\begin{split}
D\Theta_1(\varphi).v&=\left.\frac{\partial}{\partial t}\right|_{t:=0}\Xi^{\omega}(\varphi+tv)=DF^{\omega}(\varphi,T\varphi).v=\\
&=D_{.1} F^{\omega}(\varphi,T\varphi).v
+D_{.2} F^{\omega}(\varphi,T\varphi).Tv\in\Lambda((TX^{0,1})^*),
\end{split}
\end{equation}
where $v\in\mathcal{X}_{\infty}$. By Lemma \ref{lS} for any  $s\ge0$, there exists a constant $c_s>0$ and a $s_0$ with $0\le s_0<s$ such that
for all $\varphi\in R_1\cap \mathcal{X}_{\infty}$
\begin{equation}\label{inTheta_1}
\begin{split}
\|D\Theta_1(\varphi).v\|_s&\le c_s\left[\|D_{.1} F^{\omega}(\varphi,T\varphi)\|_s\|v\|_{s_0}+\|D_{.1} F^{\omega}(\varphi,T\varphi)\|_{s_0}\|v\|_s+\right.\\
&\quad+\left.\|D_{.2} F^{\omega}(\varphi,T\varphi)\|_s\|Tv\|_{s_0}+\|D_{.2} F^{\omega}(\varphi,T\varphi)\|_{s_0}\|Tv\|_s\right]\le\\
&\le K_s^1(\varphi)\|v\|_{s+1},
\end{split}
\end{equation}
where $K_s^1(\varphi)\le K_s^1(R_1)$ is a positive functional of $\varphi$ which remains bounded with $s$.
This means that $D\Theta_1(\varphi):W^s\rightarrow W^{s+1}$ is a bounded linear operator from the Sobolev space $W^s$ to $W^{s+1}$
for $\varphi\in R_1\cap X_{\infty}$ for some bounded open set of $\mathcal{X}_{0}$.
So,  $\Theta_1$ is under control.\\  For the second component, let $R_2\neq\emptyset$ be a bounded open set of $\mathcal{X}_{0}$
such that $R_1\cap R_2\neq \emptyset$, and we have
\begin{equation}
D\Theta_2(\varphi).v=\left([T\varphi_i(\varphi_i^{-1})]^{-1}.v-TA_{i,j}(\varphi_j^{-1})[T\varphi_j(\varphi_j^{-1})]^{-1}.v_j\right)_{i<j},
\end{equation}
and, as in (\ref{inTheta_1}) for any  $s\ge0$ there exists a $K_s^2(\varphi)\le K_s^2(R_2)$  positive functional of $\varphi$
which remains bounded with $s$, such that for all $\varphi\in R_2\cap \mathcal{X}_{\infty}$
\begin{equation}
\|D\Theta_2(\varphi).v\|_s\le K_s^2(\varphi)\|v\|_{s}\le K_s^2(\varphi)\|v\|_{s+1},
\end{equation}
which means that $D\Theta_2(\varphi):W^s\rightarrow W^{s+1}$ is a bounded linear operator from the Sobolev space $W^s$ to $W^{s+1}$
for $\varphi\in R_1\cap X_{\infty}$.
So,  $\Theta_2$ is under control. For the third component, let $R_3\neq\emptyset$ be a bounded open set of $\mathcal{X}_{0}$
such that $R_1\cap R_2\cap R_3\neq \emptyset$,  we have
\begin{equation}
D\Theta_3(\varphi).v=\left(\lambda(\varphi_i)D\lambda(\varphi_i).v\right)_{i<j},
\end{equation}
where
\begin{equation}
\begin{split}
\lambda(\varphi_i)&=g^{\text{FS}}\left(J((T\varphi_i)^{n-1}(\Phi_i)T\Phi_i)^{\flat}, ((T\varphi_i)^n(\Phi_i)T\Phi_i)^{\flat}\right)\\
D\lambda(\varphi_i).v&=g^{\text{FS}}\left(J((Tv)^{n-1}(\Phi_i)T\Phi_i)^{\flat}, ((T\varphi_i)^n(\Phi_i)T\Phi_i)^{\flat}\right)+\\
&\quad+g^{\text{FS}}\left(J((T\varphi_i)^{n-1}(\Phi_i)T\Phi_i)^{\flat}, ((Tv)^n(\Phi_i)T\Phi_i)^{\flat}\right).
\end{split}
\end{equation}

By Lemma \ref{lS} for any  $s\ge0$ there exists a constant $c_s>0$ and a $s_0$ with $0\le s_0<s$ such that,
\begin{equation}
\begin{split}
\|\lambda(\varphi_i)D\lambda(\varphi_i).v\|_s&\le c_s\left[\|g^{\text{FS}}\|_s\|T\varphi_i\|_{s_0}^2+\|g^{\text{FS}}\|_{s_0}\|T\varphi_i\|_{s}^2\right]\cdot\\
&\quad\cdot 2c_s\left[\|g^{\text{FS}}\|_s\|Tv\|_{s_0}\|T\varphi_i\|_{s_0}+\|g^{\text{FS}}\|_{s_0}\|Tv\|_{s}\|T\varphi_i\|_{s}\right].
\end{split}
\end{equation}
Again, as in (\ref{FS_Sob}), for complex projective varieties and their Fubini-Study metric
\begin{equation}
\|g^{\text{FS}}\|_{s}<+\infty\text{ for all }s\ge0.
\end{equation}
Hence, we infer the existence of  a $K_s^3(\varphi)\le K_s^3(R_3)$, a  positive functional of $\varphi$
which remains bounded with $s$, such that for all $\varphi\in R_3\cap \mathcal{X}_{\infty}$
\begin{equation}
\|\lambda(\varphi_i)D\lambda(\varphi_i).v\|_s\le K_s^3(\varphi_i)\|Tv\|_{s}.
\end{equation}
Therefore,
\begin{equation}
\|D\Theta_3(\varphi).(v)\|_s\le K_s^3(\varphi)\|v\|_{s+1},
\end{equation}
which means that $D\Theta_3(\varphi):W^s\rightarrow W^{s+1}$ is a bounded linear operator from the Sobolev space $W^s$ to $W^{s+1}$.\par
Since $DF^{\omega}(\varphi)$ is injective, so is $D\Theta_3$ and hence $D\Theta$. Moreover,
\begin{equation}
\|D\Theta(\varphi).v\|_s\le K_s(\varphi)\|v\|_{s+1},
\end{equation}
for $\varphi\in R\cap X_{\infty}$ for $R:=R_1\cap R_2\cap R_3\neq\emptyset$ and $K_s(\varphi):=\max_{i=1,2,3}K_s^i(\varphi)$.
By the bounded inverse operator theorem there exists a non empty  $\tilde{R}\subset R$ such that
$\Psi(\varphi):=[D\Theta(\varphi)]^{-1}$ is a bounded linear operator on $W^s$
 \begin{equation}
\|\Psi(\varphi).v\|_{s}\le \tilde{K}_{s}(\varphi)\|v\|_{s-1}\le\tilde{K}_{s}(\varphi)\|v\|_{s},
 \end{equation}
where $\tilde{K}_s(\varphi)\le \tilde{K}_s(\tilde{R})$, which remains bounded when $s$ remains bounded.
Hence, (A2) is fulfilled for $\Theta(\varphi)$ for all $\varphi\in \tilde{R}\cap X_{\infty}$.

\end{itemize}

By Theorem \ref{NMT} (i) and (ii) and the Sobolev's embedding theorem, we infer the existence of local diffeomeorphisms
 $\varphi_{i}:\Phi_{i}(U_{i})\rightarrow\mathbf{R}^{2n}$,
which lie in $C^{\infty}(\Phi_{i}(U_{i}),\mathbf{R}^{2n})$,
defining locally $B_0$ and $B_1$.
This construction is globally well defined on $X$ and leads to closed differentiable manifolds $B_0$, $B_1$ which are the boundaries of two compact differentiable submanifolds of $X$,
namely $Y_0$ and $Y_1$, such that equation (\ref{eq18}) is satisfied on every local chart.\par
Hence, $B_0$ and $B_1$ are $C^{\infty}$ submanifolds of $X$ of real codimension $0$ and $1$,
and $B_1$  is a $C^{\infty}$ submanifold of $B_0$ of real codimension $1$.
Hence, $Y_1=\partial B_1$ is a real submanifold of codimension $2$ of the complex projective manifold $X$ satisfying by construction
$\lambda^2(y)=1$ for any $y\in Y$.
By Proposition \ref{prop_lambda}, $Y$ is a complex hypersurface of the complex projective manifold $X$,
i.e. a projective submanifold of $X$.
The proof is completed.\\
\end{proof}
\begin{rem}Lemma \ref{TL} cannot be applied to any $\omega\in\Omega^{n-1,n-1}(X,\mathbf{C})$ to construct non-empty complex projective submanifolds
$Y_0$ and $Y_1$, because we find no local diffeomorphism in (\ref{omega_set})  for any $\omega\in\Omega^{n-1,n-1}(X,\mathbf{C})$.
For example for $\omega=0$ we have
\begin{equation}
\mathcal{F}^{\omega,\{U_i\}_{i=0,\dots,K}}=\emptyset .
\end{equation}
\end{rem}


\begin{lemma}\label{D0} With the same notation of Lemma \ref{TL}, if $X$ is a complex projective manifold and $B_1\subset X$ a $1$-codimensional complex submanifold,
then
\begin{equation}
D[\mu_{B_1}](\varphi)=0.
\end{equation}
\end{lemma}
\begin{proof}
Following Remark \ref{remFS} we compute the determinant of the Fubini-Study metric on $\mathbf{C}P^n$ as
\begin{equation}
\det[g^{\text{FS}}([z])]=\frac{1}{(1+\vert t_j \vert^2)^{n+1}},
\end{equation}
where $[z]\in U_j=\{[(z^1,\dots,z^j,\dots,z^n)]\vert\,z^j\neq 0\}\subset\mathbf{C}P^n$ has homogenous coordinates for the $(U_j,\Phi_j)$ chart given by
\begin{equation}
t_j:=\left(\frac{z^0}{z^j},\dots\frac{z^{j-1}}{z^j},\frac{z^{j+1}}{z^j},\dots,\frac{z^n}{z^j}\right).
\end{equation}
Following Definition \ref{sub}, if $B_1\subset\mathbf{C}P^n$ is a complex $1$-codimensional submanifold defined by specifying am holomorphic
diffeomorphisms $\{\varphi_j\}_{j=0,\dots,n}$ on an appropriate subset of $\mathbf{C}^n$, then the homogenous coordinate of a $[z]\in U_j\cap B_1$ read
\begin{equation}
t_j=\varphi^{-1}(s_1,\dots,s_{n-1},0)
\end{equation}
for a $s:=(s_1,\dots,s_{n-1})\in\mathbf{C}^{n-1}$. Therefore, the determinant of the restriction to $B_1$ of the Fubini-Study metric is
\begin{equation}
\det[g^{\text{FS}}\vert_{B_1}(\Phi_j^{-1}(\varphi_j^{-1}(s,0)))]=\frac{1}{(1+\vert s \vert^2)^{n+1}},
\end{equation}
which does not depend on $\varphi_j$. As one can see The same holds true a $1$-codimensional complex submanifolds $B_1$
of a complex projective manifold $X\subset\mathbf{C}P^n$. Just replace $\mathbf{C}P^n$ with $X$ and $\Phi_j$ with $\Phi_j\circ i_{B_1}$,
 where $i_{B_1}:B_1\rightarrow X$ is the injection, in the reasoning above.\par
We conclude that the volume form is an invariant for complex $1$-codimensional submanifolds and hence its Fr\'{e}chet derivative vanishes.
\end{proof}

\begin{lemma}\label{H} Let $\mathbf{F}\in\{\mathbf{Q},\mathbf{R},\mathbf{C}\}$ and $k\in\{0,\dots,n\}$.
 With the same notation of Lemma \ref{TL}, if $X$ is a complex projective manifold of complex dimension $n$, and $\omega\in\Omega^{n-1,n-1}(X,\mathbf{C})$
such that $[\omega]\in H^{n-1,n-1}(X,\mathbf{F})$, then
\begin{equation}\label{claimH}
[i_{\partial B_1}^*\left(\gamma^{\partial B_0}(\nu^{\partial B_1})Q^{\partial B_0}\zeta^{\partial B_1}\gamma^{X}(\nu^{\partial B_0})Q^{X}\zeta^{\partial B_0}\omega\right)]\in H^{n-1,n-1}(B_1,\mathbf{F}).
\end{equation}
\end{lemma}
\begin{proof}
We have to prove that for any $n-1$-complex dimensional $Y\subset X$ such that $[Y]\in H_{2(n-1)}(X,\mathbf{F})$ from
\begin{equation}
\int_Yi_Y^*\omega\in\mathbf{F},
\end{equation}
it follows that
\begin{equation}\label{Fcond}
\int_Wi_W^*\underbrace{\left(\gamma^{\partial B_0}(\nu^{\partial B_1})Q^{\partial B_0}\zeta^{\partial B_1}\gamma^{X}(\nu^{\partial B_0})Q^{X}\zeta^{\partial B_0}\omega\right)}_{=:\omega^1}\in\mathbf{F},
\end{equation}
for any $n-1$-complex dimensional $W\subset \partial B_1$ such that $[W]\in H_{2(n-1)}(\partial B_1,\mathbf{F})$.
The expression $\omega^1$ is a differential form in $\Omega^{n-1,n-1}(\partial B_1,\mathbf{C})$ because
\begin{itemize}
\item  the operators $\zeta^{\partial B_0}$ and $ \zeta^{\partial B_1}$ are complex bundle homomorphisms
on $\Lambda(TX^{0,1})^*|_{\partial B_0}$
 and $\Lambda(TB_0^{0,1})^*)|_{\partial B_1}$, respectively,
\item the operator $\gamma^{X}(\nu^{\partial B_0})Q^{X}$ maps $\Omega^{n-1,n-1}(X,\mathbf{C})\vert_{\partial B_0}$ into itself,
\item the operator $\gamma^{\partial B_0}(\nu^{\partial B_1})Q^{\partial B_0}$ maps $\Omega^{n-1,n-1}(\partial B_0,\mathbf{C})\vert_{\partial B_1}$ into itself.
\end{itemize}
 Since $\dim_{\mathbf{C}} \partial B_1=n-1$, there exists a complex valued $C^{\infty}$ function $c^W$ on $\partial B_1$ such that
\begin{equation}\label{ww}
i_W^*\omega^1=c^W\mu_{\partial B_1}.
\end{equation}
The differential form $\omega^1$ defines a Dolbeault cohomology class $[i_{\partial B_1}^*\omega^1]\in H^{n-1,n-1}(\partial B_1,\mathbf{C})$
because it is $\bar{\partial}$-closed. Now we can prove (\ref{Fcond}).
Let $\mathfrak{X}(W)\subset X$ the complex $0$-codimensional submanifold of $X$ such that $\partial B_1\cap \mathfrak{X}(W) \cap W=W$ and
$[\mathfrak{X}(W)]\in H_{2n}(X,\mathbf{F})$. By applying Theorem \ref{MVP} twice
and Lemma \ref{TL} we can show that
\begin{equation}\label{}
\begin{split}
&\int_{W}\alpha\wedge\overline{*}(\gamma^{\partial B_0}(\nu^{\partial B_1})Q^{\partial B_0}\zeta^{\partial B_1}\omega^1)= \int_{W}\langle\alpha(y),\gamma^{\partial B_0}(\nu^{\partial B_1})Q^{\partial B_0}_y\zeta^{\partial B_1}(y)\omega^1(y)\rangle d\text{vol}_{y\in W}=\\
&=\int_{\partial B_0\cap \mathfrak{X}(W)}\langle\alpha,\omega^1\rangle d\text{vol}_{\partial B_0}=\int_{\mathfrak{X}(W)}\alpha\wedge\overline{\ast}\omega\in\mathbf{F},
\end{split}
\end{equation}
and, by Proposition \ref{close}, which holds true only for $X$ complex projective and $\mathbf{F}\in\{\mathbf{Q},\mathbf{R},\mathbf{C}\}$,
 the proof is completed.\\
\end{proof}
\noindent We have not been proving statements about the empty set, as the following result shows.
\begin{lemma}\label{F0} With the same notation of Lemma \ref{TL}, if $X$ is a complex projective manifold and $\mathbf{F}\in\{\mathbf{Q},\mathbf{R},\mathbf{C}\}$, then it exists a
$\omega\in\Omega^{n-1,n-1}(X,\mathbf{C})$  such that
\begin{equation}
\mathcal{F}^{\omega,\{U_i\}_{i=0,\dots,K}}\neq\emptyset.
\end{equation}
\end{lemma}
\begin{proof}
We define
\begin{equation}
\omega:=\frac{w^{\wedge (n-1)}}{(n-1)!}\in\Omega^{n-1,n-1}(X,\mathbf{C}),
\end{equation}
where $w$ is the K\"ahler form on $X$. By Wirtinger's formula
\begin{equation}
i_{\partial B_1}^*\omega = \mu_{B_1},
\end{equation}
and, hence
\begin{equation}
\Xi^{\omega}(\varphi)(x)=F^{\omega}(\varphi(x),T\varphi(x))= i_{\partial B_1}^*\left(\gamma^{\partial B_0}(\nu^{\partial B_1})Q^{\partial B_0}\zeta^{\partial B_1}\gamma^{X}(\nu^{\partial B_0})Q^{X}\zeta^{\partial B_0}\omega\right)(x)-\mu_{B_1}(x).
\end{equation}
By (\ref{ww})
\begin{equation}
i_{\partial B_1}^*F^{\omega}(\varphi(x),T\varphi(x))=c^{\partial B_1}(x)\mu_{B_1}(x),
\end{equation}
for a $C^{\infty}$ function on $\partial B_1$. By Lemma \ref{D0}, we have
\begin{equation}
DF^{\omega}(\varphi,T\varphi).v= (Dc^{\partial B_1}(\varphi).v)\mu_{B_1}.
\end{equation}
Since $c^{\partial B_1}$ is not a constant functional of $\varphi$, the injectivity of $DF^{\omega}(\varphi,T\varphi)$ follows, and the proof is completed.\\
\end{proof}
\section{Proof of the Hodge Conjecture}
We want to find a basis of the rational Hodge cohomology, whose elements are fundamental classes of complex submanifolds of
the underlying complex projective manifold.
Moreover, we will see that the construction does not work for K\"ahler manifolds, and for complex projective manifolds for both
the integer Hodge cohomology and the Dolbeault cohomology.
\begin{corollary}\label{C1} Let $\mathbf{F}\in\{\mathbf{Q},\mathbf{R},\mathbf{C}\}$, and  $X$ be
 a $n$-dimensional non-singular complex projective manifold without boundary and $\omega\in\Omega^{k,k}(X,\mathbf{C})$
a representative of the cohomology class $[\omega]\in H^{k,k}(X,\mathbf{F})$ for a $k=0,\dots,n$.
For $k=1,\dots,n-1$, if there exist an atlas $\{(U_i,\Phi_i)_{i=0,K}\}$ of $X$ such that the injectivity assumption
\begin{equation}\label{condOmegaBis0}
\mathcal{F}^{\omega^{2((n-1)-k)},\{U_i\}_{i=0,\dots,K}}\neq\emptyset
\end{equation}
is satisfied, then there exists
a complex projective  submanifold $Z^k=Z^k(\omega)\subset X$ of dimension $k$ such that
\begin{equation}\label{scal}
\int_X\alpha\wedge\overline{*}\omega=\int_{Z^k}i_{Z^k}^*\alpha
\end{equation}
for all $\alpha\in\Omega^{k,k}(X,\mathbf{C})$ such that $\int_Yi_Z^*\alpha\in\mathbf{F}$ for all  complex $k$-dimensional submanifolds $Y\subset X$ such that $[Y]\in H_{2k}(X,\mathbf{F})$.\\
For $k\in\{0,n\}$ there always exists a complex projective submanifold $Z^k=Z^k(\omega)\subset X$ of dimension $k$  without requiring the injectivity assumption
(\ref{condOmegaBis0}).
\end{corollary}

\begin{proof}
First let us assume that $X$ is connected and analyze the different cases $k=0,\dots,n$:
\begin{itemize}
\item $k=0$: we can choose $Z_1^0:=\{p\}$ for a $p\in X$, because $H^{0,0}(X,\mathbf{F})=\mathbf{F}$ as Corollary 5.8 in \cite{BoTu82} carried over from the De Rham
to the Dolbeault cohomology shows, and, hence, $H^{0,0}(X,\mathbf{F})=\langle [1]\rangle_{\mathbf{F}}$ and $\omega_1:\equiv1\in\Omega^{0,0}(X,\mathbf{C})$ satisfies (\ref{scal}).
\item $k=n$: we can choose $Z_1^n:=X$, because $H^{n,n}(X,\mathbf{F})=\langle [\mu_X]/\text{Vol}(X)\rangle_{\mathbf{F}}$, where $\mu_X$ denotes the volume form on $X$,
and $\omega_1:=\mu_X/\text{Vol}(X)\in\Omega^{n,n}(X,\mathbf{C})$ satisfies (\ref{scal}).
\item $k=n-1$: let $B_0$ a $0$-real-codimensional submanifold of $X$, which has a boundary $\partial B_0$, a $1$-real-codimensional submanifold of $X$.
Let $[\omega],[\alpha]\in H^{n-1,n-1}(X,\mathbf{F})$. We apply  Theorem \ref{MVP} to obtain
\begin{equation}\label{app1}
\begin{split}
&\int_X\alpha\wedge\overline{*}\omega=\int_X\langle\alpha,\omega\rangle d\text{vol}_X=\int_{(X\setminus B_0) \cup B_0}\langle\alpha,\omega\rangle d\text{vol}_X=\\
&\quad=-\int_{\partial(X\setminus B_0)}\Big{\langle}\alpha(y),\gamma(\nu)Q_y\left(\int_{X\setminus B_0}G^{X\setminus B_0}(x,y)d\text{vol}_{x\in X\setminus B_0}\right)\omega(y)\Big{\rangle}d\text{vol}_{y\in\partial(X\setminus B_0)}+\\
&\qquad-\int_{\partial B_0}\Big{\langle}\alpha(y),\gamma(\nu)Q_y\left(\int_{B_0}G^{B_0}(x,y)d\text{vol}_{x\in B_0}\right)\omega(y)\Big{\rangle}d\text{vol}_{y\in\partial B_0}=\\
&\quad=\int_{\partial B_0}\langle\alpha(y),\gamma(\nu)Q_y\zeta(y)\omega(y)\rangle d\text{vol}_{y\in\partial B_0},
\end{split}
\end{equation}
where
\begin{itemize}
\item the hermitian structure in the antiholomorphic bundle over $X$ as in Proposition \ref{prop21} is denoted by $\langle\cdot,\cdot\rangle:=\cdot\wedge\overline{*}\cdot$,
\item the Dirac operator on $X$ is denoted by $Q$,
\item the Green functions for the Hodge-Kodaira Laplacians on $X\setminus B_0$ and $B_0$ as in Proposition \ref{prop31} are denoted by $G^{X\setminus B_0}$, and,
respectively by $G^{B_0}$,
\item the operator
\begin{equation}
\zeta(y):=\left(\int_{X\setminus B_0}G^{X\setminus B_0}(x,y)+\int_{B_0}G^{B_0}(x,y)\right)d\text{vol}_{x\in X}
\end{equation}
is a complex bundle homomorphism on $\Lambda(TX^{0,1})^*|_{\partial B_0}$.
\end{itemize}
Note that the inward unit normal fields on the boundaries of $X\setminus B_0$ and $B_0$ are in opposite directions.\par
Let now $B_1$ be $0$-real-codimensional submanifold of $\partial B_0$, and let us apply Theorem \ref{MVP} a second time to (\ref{app1}) and obtain

\begin{equation}\label{app2}
\begin{split}
\int_X\alpha\wedge\overline{*}\omega&=\int_{\partial B_0}\langle\alpha,\omega^1\rangle d\text{vol}_{\partial B_0}=\\
&=\int_{\partial B_1}\langle\alpha(y),\gamma^{\partial B_0}(\nu^{\partial B_1})Q^{\partial B_0}_y\zeta^{\partial B_1}(y)\omega^1(y)\rangle d\text{vol}_{y\in \partial B_1}=\\
&=\int_{\partial B_1}\alpha\wedge\overline{*}(\gamma^{\partial B_0}(\nu^{\partial B_1})Q^{\partial B_0}\zeta^{\partial B_1}\omega^1),
\end{split}
\end{equation}
where

\begin{itemize}
\item the differential form $\omega^1(y):=\gamma^{X}(\nu^{\partial B_0})Q^{X}_y\zeta^{\partial B_0}(y)\omega(y)
\in\Omega^{k,k}(\partial B_0,\mathbf{C})$ is defined for $y\in\partial B_0$
\item the hermitian structure in the antiholomorphic bundle over $\partial B_0$ as in Proposition \ref{prop21} is denoted by $\langle,\cdot,\cdot\rangle$,
\item the Dirac operator on $\partial B_0$ is denoted by $Q^{\partial B_0}$,
\item the Green functions for the Hodge-Kodaira Laplacians on $\partial B_0\setminus B_1$ and $B_1$ as in Proposition \ref{prop31} are denoted by $G^{X\setminus\partial B_0}$, and,
respectively by $G^{\partial B_0}$,
\item the operator
\begin{equation}
\zeta^{\partial B_0}(y):=\left(\int_{\partial B_0\setminus B_1}G^{\partial B_0\setminus B_1}(x,y)+\int_{B_1}G^{B_1}(x,y)\right)d\text{vol}_{x\in\partial B_0}
\end{equation}
is a complex bundle homomorphism on $\Lambda(T\partial B_0^{0,1})^*|_{\partial B_1}$.
\end{itemize}
We look for $B_1$ such that for all $\alpha$ with $[\alpha]\in H^{n-1,n-1}(X,\mathbf{F})$
\begin{equation}\label{app220}
\int_X\alpha\wedge\overline{*}\omega=\int_{\partial B_1}\alpha\wedge\overline{*}(\gamma^{\partial B_0}(\nu^{\partial B_1})Q^{\partial B_0}\zeta^{\partial B_1}\omega^1)=
\int_{\partial B_1}i_{\partial B_1}^*(\alpha),
\end{equation}
which can hold true only if
\begin{equation}
i_{\partial B_1}^*\left(\gamma^{\partial B_0}(\nu^{\partial B_1})Q^{\partial B_0}_y\zeta^{\partial B_1}(y)\omega^1\right)=\bar{*}1=\mu_{\partial_{B_1}},
\end{equation}
which is equivalent to
\begin{equation}\label{app2222}
i_{\partial B_1}^*\left(\gamma^{\partial B_0}(\nu^{\partial B_1})Q^{\partial B_0}\zeta^{\partial B_1}\gamma^{X}(\nu^{\partial B_0})Q^{X}\zeta^{\partial B_0}\omega\right)=\mu_{\partial_{B_1}},
\end{equation}
By (\ref{condOmegaBis0}) we can now apply Lemma \ref{TL} to solve equation (\ref{app2222}) to find the submanifolds $B_0$ and $B_1$.
 The complex submanifold $Z=Z^{n-1}:=\partial B_1(\omega)$ has complex dimension $n-1$ of $X$ and is complex projective  manifold.

\item $k=n-2,\dots,1$: for $[\omega],[\alpha]\in H^{k,k}(X,\mathbf{F})$ we continue applying Theorem \ref{MVP} till a complex $k$-codimensional submanifold appears:
\begin{equation}\label{app3}
\begin{split}
\int_X\alpha\wedge\overline{*}\omega&=\int_{\partial B_0}\langle\alpha(y),\gamma(\nu)Q_y\zeta(y)\omega(y)\rangle d\text{vol}_{y\in\partial B_0}=\\
&=\int_{\partial B_1}\langle\alpha(y),\gamma^{\partial B_0}(\nu^{\partial B_1})Q^{\partial B_0}_y\zeta^{\partial B_1}(y)\omega^1(y)\rangle d\text{vol}_{y\in\partial B_1}=\\
&=\dots=\\
&=\int_{\partial B_{2k-1}}\langle\alpha(y),\gamma^{\partial U_{2k-2}}(\nu^{\partial U_{2k-1}})Q_y^{\partial U_{2k-2}}\zeta^{\partial U_{2k-1}}(y)\omega^{2k-1}(y) \rangle d\text{vol}_{y\in\partial U_{2k-1}}\\
&=\int_{\partial B_{2k-1}}\alpha\wedge \overline{*}(\gamma^{\partial U_{2k-2}}(\nu^{\partial U_{2k-1}})Q^{\partial U_{2k-2}}\zeta^{\partial U_{2k-1}}\omega^{2k-1}),
\end{split}
\end{equation}
where
\begin{itemize}
\item the submanifolds $B_0,B_1,\dots,B_{2k-1}$ of $X$ have real codimensions $0, 1, 2,\dots, 2k-1$,
\item the differential form $\omega^{2k-1}(y):=\gamma^{\partial B_{2k-3}}(\nu^{\partial B_{2k-2}})Q^{\partial B_{2k-3}}_y\zeta^{\partial B_{2k-3}}(y)\omega^{2k-3}(y)
\in\Omega^{k,k}(\partial B_{2k-1},\mathbf{C})$
is defined for $y\in\partial B_{2k-1}$,
\item the Dirac structure on $\partial B_k$ induced by the Dirac structure on $B_{k-1}$ by Theorem \ref{DiracBoundary} has Dirac operator $Q^{\partial B_k}$,
\item the Green functions for the Hodge-Kodaira Laplacians on $\partial B_k\setminus B_{k+1}$ and $B_k$ as in Proposition \ref{prop31} are denoted by $G^{B_k\setminus\partial B_{k+1}}$, and,
respectively by $G^{\partial B_k}$,
\item the operator
\begin{equation}
\zeta^{\partial B_k}(y):=\left(\int_{\partial B_k\setminus B_{k+1}}G^{\partial B_k\setminus B_{k+1}}(x,y)+\int_{B_k}G^{B_{k+1}}(x,y)\right)d\text{vol}_{x\in\partial B_k}
\end{equation}
is a complex bundle homomorphism on $\Lambda(T\partial B_k^{0,1})^*|_{\partial B_{k+1}}$.
\end{itemize}
We look for $B_{2k-1}$ such that for all $\alpha$ with $[\alpha]\in H^{k,k}(X,\mathbf{F})$
\begin{equation}\label{app2200}
\begin{split}
\int_X\alpha\wedge\overline{*}\omega&=\int_{\partial B_{2k-1}}\alpha\wedge\overline{*}(\gamma^{\partial B_{2k-2}}(\nu^{\partial B_{2k-1}})Q^{\partial B_{2k-2}}\zeta^{\partial B_{2k-1}}\omega^{2k-1})=\\
&=\int_{\partial B_{2k-1}}i_{\partial B_1}^*(\alpha),
\end{split}
\end{equation}
which can hold true only if
\begin{equation}\label{equivapp}
i_{\partial B_{2k-1}}^*\left(\gamma^{\partial B_{2k-2}}(\nu^{\partial B_{2k-1}})Q^{\partial B_{2k-2}}\zeta^{\partial B_{2k-1}}\omega^{2k-1}\right)=\bar{*}1=\mu_{\partial_{B_{2k-1}}}.
\end{equation}
which is equivalent to
\begin{equation}\label{app4}
\begin{split}
&i_{\partial B_{2k-1}}^*\left(\gamma^{\partial B_{2k-2}}(\nu^{\partial B_{2k-1}})Q^{\partial B_{2k-2}}\zeta^{\partial B_{2k-1}}\gamma^{\partial B_{2k-3}}(\nu^{\partial B_{2k-2}})Q^{\partial B_{2k-3}}\zeta^{\partial B_{2k-2}}\omega^{2k-2}\right)=\\
&\quad=\mu_{\partial B_{2k-1}}.
\end{split}
\end{equation}

Equation (\ref{app4}) has been solved for $k=n-1$. Assuming that its has been solved for $k-1$, the differential form $\omega^{2k-2}$ is well defined,
By (\ref{condOmegaBis0}) we can now apply Lemma \ref{TL} to solve equation (\ref{app4}) to find the submanifolds $B_{2k-1}$ and $B_{2k-2}$, for any $k=n-2,\dots,1$.
 The complex submanifold $Z=Z^k:=\partial B_{2k-1}(\omega)$ of $X$ has complex dimension $k$ and is a complex projective manifold.

\end{itemize}
If $X$ is not connected, then it can represented as disjoint union of its connected components $(X_{\iota\in I})$. Since for any $k=0,\dots,n$
\begin{equation}
\Omega^{k,k}(X,\mathbf{C})=\bigoplus_{\iota\in I}\Omega^{k,k}(X_{\iota},\mathbf{C}) \quad\text{ and }\quad H^{k,k}(X,\mathbf{F})=\bigoplus_{\iota\in I}H^{k,k}(X_{\iota},\mathbf{F}),
\end{equation}
the result follows from the connected case and the proof is completed.\\
\end{proof}
%
\begin{rem} Without the assumption (\ref{condOmegaBis0}) Corollary \ref{C1} cannot hold for all $\omega\in H^{k,k}(X,\mathbf{F})$
as the simple counterexample $\omega:=0$ shows.
\end{rem}

\noindent Corollary \ref{C1}, reformulated using Definition \ref{deffc}, leads to

\begin{corollary}\label{C1bis}
Let $\mathbf{F}\in\{\mathbf{Q},\mathbf{R},\mathbf{C}\}$, $X$ be a $n$-dimensional complex projective manifold without boundary and $\omega\in\Omega^{k,k}(X,\mathbf{C})$
a representative of the cohomology class $[\omega]\in H^{k,k}(X,\mathbf{F})$ for a $k=1,\dots,n-1$. Then, $\bar{*}[\omega]$ is a fundamental class
of a closed complex projective  submanifold of complex codimension $n-k$ if and only if
there exist an atlas $\{(U_i,\Phi_i)_{i=0,K}\}$ of $X$ such that
\begin{equation}\label{condOmegaBis00}
\mathcal{F}^{\omega^{2((n-1)-k)},\{U_i\}_{i=0,\dots,K}}\neq\emptyset.
\end{equation}
Moreover, for $k=0,n$ we have $H^{0,0}(X,\mathbf{F})=\langle [1]\rangle_{\mathbf{F}}$ and $H^{n,n}(X,\mathbf{F})=\langle [\mu_X]/\text{Vol}(X)\rangle_{\mathbf{F}}$,
where $\mu_X$ denotes the volume form on $X$.
\end{corollary}

\begin{lemma}\label{L_dense} Let $\mathbf{F}\in\{\mathbf{Z}, \mathbf{Q},\mathbf{R},\mathbf{C}\}$,  $X$ be a $n$-dimensional non-singular complex projective manifold
without boundary and for any $k=1,\dots,n-1$,
\begin{equation}\label{A_def}
A^k(X,\mathbf{F}):=\left\{\omega\in\Omega^{k,k}(X,\mathbf{C})\,\big{\vert}\,[\omega]\in H^{k,k}(X,\mathbf{F}),
\text{ such that }\mathcal{F}^{\omega^{2((n-1)-k)},\{U_i\}_{i=0,\dots,K}}\neq\emptyset\right\}.
\end{equation}
Then, with respect to the $L^2$ norm for differential forms, the finite linear hull $\langle A^k(X,\mathbf{Q})\rangle_{\mathbf{Q}}$ is dense in
\begin{equation}
\begin{split}
\text{Hdg}^k(X,\mathbf{Q})&\cong\left[\ker(\Delta^{k,k}_{\bar{\partial}})\cap\Bigg{\{}\omega\in\Omega^{k,k}(X,\mathbf{C})\,\Big{\vert}\,\int_Yi_Y^*\omega\in\mathbf{Q}\right.\\
&\left.\qquad\text{ for all $k$ dimensional  }Y\subset X\text{ such that }[Y]\in H_{2k}(X,\mathbf{Q})\Bigg{\}}\right].
\end{split}
\end{equation}
\end{lemma}
\begin{proof} We first provide a proof for $k=n-1$.
In Lemma \ref{TL} we saw
\begin{equation}
\Xi^{\omega}(\varphi):= \gamma^{\partial B_0}(\nu^{\partial B_1})Q^{\partial B_0}\zeta^{\partial B_1}\gamma^{X}(\nu^{\partial B_0})Q^{X}\zeta^{\partial B_0}\omega-\mu_{\partial B_1}
=F^{\omega}(\varphi,T\varphi),
\end{equation}
where $\varphi:=(\varphi_0,\varphi_1,\dots,\varphi_K)$, and
\begin{equation}
\Xi^{\omega}(\varphi)(x)= F^{\omega}(\varphi(x),T_x\varphi).
\end{equation}
 Therefore,
\begin{equation}
DF^{\omega}(\varphi,T\varphi)=\underbrace{D[\gamma^{\partial B_0}(\nu^{\partial B_1})Q^{\partial B_0}\zeta^{\partial B_1}\gamma^{X}(\nu^{\partial B_0})Q^{X}\zeta^{\partial B_0}]}_{=:a(\varphi,T\varphi)}\omega+\underbrace{D[-\mu_{\partial B_1}]}_{=0\text{ by Lemma }\ref{D0}}
\end{equation}
is a linear functional of $\omega\in\Omega^{k,k}(X,\mathbf{C})$, and for all $x\in X$ and $v\in\mathcal{X}_{\infty}$
\begin{equation}
DF^{\omega}(\varphi(x),T_x\varphi).v(x)=(a(\varphi(x),T_x\varphi).v(x))\omega.
\end{equation}
For any $\omega\neq 0$ with $[\omega]\in H^{k,k}(X,\mathbf{Q})$ and any collection of holomorphic diffeomorphisms $\varphi$
satisfying the compatibility condition, the linear operator $DF^{\omega}(\varphi,T\varphi)$ is injective, as we prove now. Let us suppose that
\begin{equation}\label{inj_v}
DF^{\omega}(\varphi(x),T_x\varphi).v(x)=0
\end{equation}
for $v\in\mathcal{X}_{\infty}$. Since $[\omega]\in H^{k,k}(X,\mathbf{Q})$ by Lemma \ref{H}
\begin{equation}\label{disc}
\int_W i_W^*\left(\gamma^{\partial B_0}(\nu^{\partial B_1})Q^{\partial B_0}\zeta^{\partial B_1}\gamma^{X}(\nu^{\partial B_0})Q^{X}\zeta^{\partial B_0}\omega\right)\in \mathbf{Q},
\end{equation}
for any $k$-dimensional complex submanifold $W\subset B_1$ such that $[W]\in H_{2k}(B_1,\mathbf{Q})$. By (\ref{inj_v}) we obtain
\begin{equation}
\int_Wi_Y^*(DF^{\omega}(\varphi,T\varphi).v)=0,
\end{equation}
which can only be true at the same time as (\ref{disc}) if and only if $v=0$. Hence, the injectivity of $DF^{\omega}(\varphi,T\varphi)$
is proved in the rational cohomology case for $k=n-1$. This proof cannot be extended to the real or complex cohomology case.\\
For $k=n-2,\dots,1$ we follow the passing through dimension method (\ref{app3}) explained in the proof of Corollary \ref{C1}, and apply
Lemmata \ref{D0}, \ref{H} and \ref{F0} to:
\begin{itemize}
\item the submanifold $B_{2k-1}$ instead of $B_0$,
\item the submanifold $B_{2k-2}$ instead of $B_1$
\item the operator $\gamma^{\partial B_{2k-2}}(\nu^{\partial B_{2k-1}})Q^{\partial B_{2k-2}}\zeta^{\partial B_{2k-1}}\gamma^{\partial B_{2k-3}}(\nu^{\partial B_{2k-2}})Q^{\partial B_{2k-3}}\zeta^{\partial B_{2k-2}}$
instead of the operator
$\gamma^{\partial B_0}(\nu^{\partial B_1})Q^{\partial B_0}\zeta^{\partial B_1}\gamma^{X}(\nu^{\partial B_0})Q^{X}\zeta^{\partial B_0}$,
\item the differential form $\omega^{2k-2}$ on $B_{2k-2}$ instead of $\omega$ on $X$.
\end{itemize}
Let us consider
\begin{equation}\label{omega_set_bis}
\begin{split}
&\mathcal{F}^{\omega,\{U_i\}_{i=0,\dots,K}}=\left\{\psi:=(\psi_0,\psi_1,\dots,\psi_K)|\;\psi_i:\Phi_i(U_i)\rightarrow\psi_i(\Phi_i(U_i))\text{ is a } \right.\\
&\qquad\qquad\qquad\qquad\left.\text{diffeomorphism for all }i=0,\dots,K\text{ and }DF^{\omega}(\psi,T\psi)\text{ is injective}\right.\}.
\end{split}
\end{equation}and assume that $A^k(X,\mathbf{Q})$ 
is \textit{not} dense in
$\text{Hdg}^k(X,\mathbf{Q})$. Since for the zero Hodge class $[0]\in\text{Hdg}^k(X,\mathbf{Q})$ $DF^{0}(\varphi(x),T_x\varphi)=0$, and
$\mathbf{Q}$ is dense in $\mathbf{R}$, there exists a positive \textit{rational} constant $\epsilon>0$  such that
\begin{equation}
B_{\epsilon}^{\text{Hdg}^k(X,\mathbf{Q})}(0)\cap A^k(X,\mathbf{Q}) =\emptyset,
\end{equation}
where $B_{\epsilon}^{\text{Hdg}^k(X,\mathbf{Q})}(0)$ is the closed $L^2$ ball of center $0$ and radius $\epsilon$ in $\text{Hdg}^k(X,\mathbf{Q})$.
Any $\omega\in\Omega^{k,k}(X,\mathbf{C})$ such that $[\omega]\in\text{Hdg}^k(X,\mathbf{Q})$ and $\omega\ne 0$ can be written as
\begin{equation}
\omega = \underbrace{\left[\frac{C}{\epsilon}\right]}_{\in\mathbf{Q}_+}\underbrace{\left[\epsilon\frac{\omega}{C}\right]}_{\in B_{\epsilon}^{\text{Hdg}^k(X,\mathbf{Q})}(0)},
\end{equation}
for a rational $C\le\|\omega\|_{L^2}$. Now, we have
\begin{equation}
DF^{\omega}(\varphi(x),T_x\varphi).v(x) = \left[\frac{C}{\epsilon}\right]DF^{\epsilon\frac{\omega}{C}}(\varphi(x),T_x\varphi).v(x),
\end{equation}
and $DF^{\omega}(\varphi(x),T_x\varphi)$ cannot be injective because $DF^{\epsilon\frac{\omega}{C}}(\varphi(x),T_x\varphi)$ is not.
Note that this reasoming cannot be carried over to the integer cohomology case, because $\mathbf{Z}$ is not dense in $\mathbf{R}$.
Since $\omega\neq 0$, it follows that
\begin{equation}
A^k(X,\mathbf{Q})\cap\text{Hdg}^k(X,\mathbf{Q})=\emptyset,
\end{equation}
which cannot be true, because for any $\omega\neq 0$ with $[\omega]\in H^{k,k}(X,\mathbf{Q})$ and any collection of holomorphic diffeomorphisms $\varphi$ satisfying the compatibility condition, the linear operator $DF^{\omega}(\varphi,T\varphi)$ is injective, as we have shown above.\\
The proof is completed and does not hold for the integer, real or complex cohomology case.\\
\end{proof}

\begin{theorem}\label{thmH}
Conjecture \ref{Hodge2} holds true for any complex projective manifold $X$. More exactly,
there exist $Q:=\dim_{\mathbf{Q}} (\text{Hdg}^k(X))$ $k$-codimensional submanifolds of $X$, $Z_1,\dots,Z_Q$, such that
\begin{equation}
\text{Hdg}^k(X,\mathbf{Q})=\langle[Z_1],\dots,[Z_Q]\rangle_\mathbf{Q}.
\end{equation}
\end{theorem}
\begin{proof}
For $k\in\{0,n\}$ it follows directly from Corollary \ref{C1bis}. For $k=1,\dots,n-1$,
the $2k$-Hodge class group, defined as
\begin{equation}
\text{Hdg}^k(X,\mathbf{Q}):=H^{2k}(X,\mathbf{Q})\cap H^{k,k}(X,\mathbf{C}),
\end{equation}
by Proposition \ref{close}, being $\bar{*}$ an isomorphismus, can be represented as

\begin{equation}
\text{Hdg}^k(X,\mathbf{Q})=\langle\bar{*}[\omega_1],\dots,\bar{*}[\omega_Q]\rangle_\mathbf{Q},
\end{equation}
where $Q:=\dim_{\mathbf{Q}}(\text{Hdg}^k(X))$ and $\omega_1,\dots,\omega_Q$ are rational $(n-k,n-k)$ differential forms on $X$, i.e.
\begin{equation}
\int_{Y}i_Y^*(\omega_m)\in\mathbf{Q}
\end{equation}
for all complex $k$-codimensional submanifolds $Y$ of $X$, such that $[Y]\in H_{2n-2k}(X,\mathbf{Q})$ and all $m=1,\dots,Q$.
Since $\bar{*}$ sends harmonic $(n-k,n-k)$-forms to harmonic $(k,k)$-forms, if $\{\omega_1,\dots\omega_Q\}$ are harmonic, by Lemma \ref{L_dense} we have
\begin{equation}
\langle\bar{*}[\omega_1],\dots,\bar{*}[\omega_Q]\rangle_\mathbf{Q}=\langle A^k(X,\mathbf{Q})\rangle_{\mathbf{Q}},
\end{equation}
Every cohomology class has a unique harmonic representative. By Corollary \ref{C1bis} and Definition \ref{deffc} we can choose
the rational harmonic differential forms $\{\omega_1,\dots\omega_Q\}$  so that
 there exist $Q$  $k$-codimensional complex submanifolds $Z_1(\omega_1),\dots,Z_Q(\omega_Q)$
of $X$ such
\begin{equation}
\bar{*}[\omega_m]=[Z_m],
\end{equation}
for all $m=1,\dots,Q$, and, hence
\begin{equation}
\text{Hdg}^k(X,\mathbf{Q})=\langle[Z_1],\dots,[Z_Q]\rangle_\mathbf{Q}.
\end{equation}
as  Conjecture \ref{Hodge2} states.\\
\end{proof}
\begin{rem} The statement of Theorem \ref{thmH} is actually slightly stronger than the original Hodge conjecture, Conjecture \ref{Hodge2},
for it constructs a representation of the rational Hodge cohomology $\text{Hdg}^k(X,\mathbf{Q})$ as rational linear combination of
fundamental classes $[Z_1],\dots,[Z_Q]$ of complex submanifolds $Z_1,\dots,Z_Q$, not just subvarieties of the complex projective manifold $X$.
This means that $Z_1,\dots,Z_Q$ have no singularities.
\end{rem}
\begin{rem}
As we saw in its proof Lemma \ref{L_dense} cannot be extended to the integer Hodge cohomology, which is consistent
with the fact that the Hodge conjecture with integer coefficients is not true, as the
counterexamples of Atiyah-Hirzebruch \cite{AtHi62} and Totaro \cite{To97} demonstrate.
Therefore, the proof of Theorem \ref{thmH} does not extend to integer cohomology.
Moreover, Lemmata \ref{D0} and \ref{H} which are essential in the proof of Lemma \ref{L_dense}, strongly rely
on the complex projective manifold structure of $X$ and on its Riemannian metric given by the restriction of
the Fubini-Study metric. Hence, the proof of \ref{thmH} does not extend to K\"ahler manifolds, which is in line with the counterexamples
of Zucker \cite{Zu77} and Voisin \cite{Vo02}.
\end{rem}
\begin{rem}
The Hodge classes $[Z_1],\dots,[Z_Q]$ can be completed to a complex basis of the Dolbeault cohomology $H^{k,k}(X,\mathbf{C})$, by adding appropriate complex linear independent
cohomology classes $[\eta_{Q+1}],\dots,[\eta_C]$, which, however, are not fundamental classes of $X$.
\end{rem}
\noindent From Theorem \ref{thmH} we can now infer the validity of the Hodge conjecture.
\begin{theorem}
Conjecture \ref{Hodge1} holds true for any non singular projective algebraic variety.
\end{theorem}

\section{Conclusion}
A K\"ahler manifold can be seen as a Riemannian manifold carrying a Dirac bundle structure whose Dirac operator is the Dirac-Dolbeault operator.
Utilizing a theorem for the Green function for the Dirac Laplacian over a Riemannian manifold with boundary, the values of the sections of the Dirac bundle can be
represented in terms of the values on the boundary, extending the mean value theorem of harmonic analysis.
This representation, together with Nash-Moser generalized
inverse function theorem, leads to a technical result stating the existence of complex submanifolds of a projective manifold
 satisfying globally a certain partial differential equation under a certain injectivity assumption. This is the key to prove the existence
of complex submanifolds of a complex projective manifold
whose fundamental classes span the rational Hodge classes, proving the Hodge conjecture for non singular algebraic varieties.

\section*{Acknowledgement}
I would like to express my deep gratitude to Roberto Ferretti and Dustin Clausen for the many discussions which lead to improvements and reformulations of the present paper,
I would like to thank Claire Voisin and Pierre Deligne for highlighting parts of the first version of this paper which
 needed  important corrections. The possibly remaining mistakes are all mine.

\end{document}